%% file: 2016crelle.tex
\documentclass[a4paper,12pt,leqno]{amsart}
\usepackage{amssymb}
\usepackage{amsmath}
\usepackage{amscd}
\usepackage{stmaryrd}
\usepackage{amsbsy}

\usepackage[dvips,final]{graphicx}
\usepackage[dvips]{geometry}
\usepackage{color} %include even if images aren???t in color
\usepackage{epsfig}
\usepackage{latexsym}
\usepackage{pstricks}

\usepackage%[UglyObsolete,tight,heads=LaTeX]
{diagrams}

\def\vs{\vskip .6cm}

\def\ra{\rightarrow}
\def\nb{\nabla}

\def\G{\Gamma}

\def\a{\alpha}

\def\c{\gamma}
\def\e{\varepsilon}
\def\f{\varphi}
\def\k{\kappa}
\def\l{\lambda}
\def\n{\nabla}
\def\o{\omega}
\def\s{\sigma}
\def\t{\tilde}
\def\tl{\tilde}

\def\beq{\begin{equation}}
\def\eeq{\end{equation}}
\def\bi{\begin{enumerate}}
\def\ei{\end{enumerate}}
\def\bea{\begin{eqnarray*}}
\def\eea{\end{eqnarray*}}
\def\ba{\begin{array}}
\def\ea{\end{array}}

\def\r{\end{proof}}

\def\I{\mathcal{I}}
\def\A{\mathcal{A}}
\def\B{\mathcal{B}}

\def\U{\mathcal{U}}

\def\ll{\mathcal{L}}

\def\sm{\smallsetminus}

\def \R{\mathbb{R}}
\def \RM{\mathbb{R}}
\def \Q{\mathbb{Q}}
\def \N{\mathbb{N}}
\def \Z{\mathbb{Z}}
\def \C{\mathbb{C}}

\def\d{{\partial}}
\def\Ric{\mathrm{Ric}}
\def\id{\mathrm{Id}}
\def\be{\begin{equation}}
\def\ee{\end{equation}}

\def\SO{\mathrm{SO}}
\def\Sp{\mathrm{Sp}}
\def\SU{\mathrm{SU}}
\def\UU{\mathrm{U}}
\def\Spin{\mathrm{Spin}}
\def\Hol{\mathrm{Hol}}
\def\mm{M_0}
\def\mmm{\widehat{\mm}}
\def\dt{\delta}

\newtheorem{epr}{Proposition}[section]

\newtheorem{prop}[epr]{Proposition}

\newtheorem{ath}[epr]{Theorem}
\newtheorem{elem}[epr]{Lemma}
\newtheorem{ecor}[epr]{Corollary}
\newtheorem{conj}[epr]{Conjecture}

\theoremstyle{definition}
\newtheorem{defi}[epr]{Definition}
\newtheorem{ede}[epr]{Definition}
\newtheorem{ere}[epr]{Remark}
\newtheorem{ex}[epr]{Example}

\def\obs{\begin{ere}}
\def\eobs{\end{ere}}
\def\bl\begin{elem}
\def\el\end{elem}
\def\bdf\begin{ede}
\def\edf\end{ede}
 
\title{On the irreducibility of locally metric connections}
\author{Florin Belgun, Andrei Moroianu}

\address{Florin Belgun\\Fachbereich Mathematik\\
Bereich AD\\
Bundesstr. 55 (Geomatikum)\\
20146 Hamburg, Germany}
 \email{florin.belgun@math.uni-hamburg.de}

\address{Andrei Moroianu \\ Universit\'e de Versailles-St Quentin \\
Laboratoire de Math\'ema\-tiques \\ UMR 8100 du CNRS\\
45 avenue des \'Etats-Unis\\
78035 Versailles, France }
\email{andrei.moroianu@math.cnrs.fr}

\thanks{This work was partially supported by the ANR-10-BLAN 0105 grant of the
Agence Nationale de la Recherche.}

\begin{document}

\begin{abstract}
A locally metric connection on a smooth manifold $M$ is a torsion-free
connection $D$ on $TM$ with 
compact restricted holonomy group $\Hol_0(D)$. 
If the holonomy representation of such a connection is irreducible,
then $D$ preserves a conformal structure on $M$. Under some natural
geometric assumption on the life-time of incomplete geodesics, we
prove that conversely, a locally metric connection $D$ preserving a
conformal structure on a compact manifold $M$ has irreducible holonomy
representation, unless $\Hol_0(D)=0$ or $D$ is the
Levi-Civita connection of a Riemannian metric
on $M$. This result generalizes Gallot's theorem on the
irreducibility of Riemannian cones to a much wider class of connections. As
an application, we give the geometric description of compact conformal manifolds carrying a tame
closed Weyl connection with non-generic holonomy. 
 
% 
% A classical theorem of Gallot \cite{gal} states that a Riemannian cone over a
% compact manifold is either irreducible or flat. Such a cone has compact
% quotients by radial homotheties (which form a 1-parameter group). More
% generally, we define {\em cone-like} manifolds to be those non-compact manifolds
% that admit compact quotients by discrete subgroups of homotheties and show that,
% under some {\em tameness} assumption (concerning the life-time of incomplete
% geodesics), all cone-like manifolds are either irreducible or flat. This
% assumption holds, in particular, for any small cone-like deformation of
% Riemannian cones.
% 
% Using the natural correspondence between cone-like manifolds and compact
% conformal manifolds with a closed Weyl connection, our result can be restated as
% follows: Every closed, non-exact, {\em tame} Weyl connection on a compact
% conformal manifold is either flat, or has irreducible holonomy. As an
% application, we describe the compact conformal manifolds carrying a tame closed
% Weyl connection with non-generic holonomy. 
  \bigskip


\end{abstract}

\maketitle

\vs
\section{Introduction}

The restricted holonomy group of the Levi-Civita
connection of a Riemannian manifold $(M,g)$ is (conjugate to) a closed
subgroup of $\SO(n)$, and thus compact. This property actually
characterizes locally the Levi-Civita connections: If $D$ is a
torsion-free connection on some manifold $M$ with compact restricted
holonomy group $\Hol_0(D)$, then every point $x$ of $M$ has a neighborhood
with a Riemannian metric on it whose Levi-Civita connection is
$D$. 
% To see this, pick any scalar product on $T_xM$, take its mean value over 
% $\Hol_0(D)\subset \mathrm{End}(T_xM)$ with respect to the Haar measure to obtain 
% a $\Hol_0(D)$-invariant scalar product $g_x$ on $T_xM$, then extend it by parallel transport
% on a simply connected neighborhood of $x$.
This explains the following:
\begin{ede}
A {\em locally metric connection} on a smooth manifold $M$ is a
torsion-free connection $D$ on $TM$ with compact restricted holonomy
group $\Hol_0(D)$. 
\end{ede}

A typical example of locally (but not globally) metric connection is
the following:  

\begin{ex}\label{ex1} Let $(N,g^N)$ be
a Riemannian manifold, $\lambda>1$ a real number and consider the
product $\bar N:=\R^*_+\times N.$ The group $\Gamma$ generated by the
dilation $\gamma:\bar N\to \bar N$, $\gamma(t,x)=(\lambda t,x)$
consists of strict homotheties of the cone metric $\bar
g:=dt^2+t^2g^N$ on $\bar N$, so the Levi-Civita connection of $\bar g$
induces a locally metric connection $D$ on $M:=\bar N/\Gamma=S^1\times
N$ which is not globally metric (see also Example \ref{leit} below). 
\end{ex}

In this paper we study the reducibility question for locally metric
connections. Let us first remark that an irreducible locally metric
connection $D$ always preserves a conformal structure on $M$. Indeed,
the lift $\tilde D$ of $D$ to the universal cover $\tilde M$ of $M$ has compact holonomy group 
$\Hol(\tilde D)=\Hol_0(D)$, thus is the Levi-Civita connection of a Riemannian metric $\tilde g$ on $\tilde M$.
Moreover, the fundamental group $\Gamma$ of $M$ acts on $\tilde M$ by $\tilde D$-affine transformations, hence by
homotheties (for every $\gamma\in\Gamma$ the metric $\gamma^*g$ is $\tilde D$-parallel, thus homothetic to $g$ by the irreducibility hypothesis).
Consequently, the conformal structure $[\tilde g]$ defined by $\tilde g$ is $\Gamma$-invariant, so it defines a $D$-parallel conformal structure on $M$.

Conversely, one can ask whether a locally metric connection $D$
preserving a conformal structure on $M$ is necessarily irreducible. The
answer is negative in general, as shown by the following examples:
\begin{itemize}
 \item[(a)] $D$ is the Levi-Civita connection of a product of
   Riemannian manifolds; 
 \item[(b)] $M$ is the quotient of $\RM^n\setminus\{0\}$ by the group
   $\Gamma$ generated by multiplication with some $\lambda>1$ and $D$
   is the connection induced on $(\RM^n\setminus\{0\})/\Gamma$ by the
   Levi-Civita connection of the flat metric of $\RM^n\setminus\{0\}$;
\item[(c)] $D$ is obtained like in Example \ref{ex1} from the product
  of two Riemannian cones (which is itself a Riemannian cone, see \cite{al}).
\end{itemize}

Remark now that in case (a) $D$ is globally metric, in case (b) $D$ is
flat, whereas in case (c) the manifold $M$ is non-compact. It turns
out that all known examples of reducible locally metric connections
preserving a conformal structure fall in one of the three cases
above. It is therefore tempting to make the following:

\begin{conj}\label{conj}
A locally metric conformal connection on a compact manifold $M$ which
is neither flat, nor globally metric, has irreducible holonomy. 
\end{conj}

In spite of its simplicity, this statement is still open in full
generality. Some evidence is provided by the beautiful theorem of Gallot
\cite{gal} concerning the irreducibility of Riemannian cone metrics,
which is the core of B\"ar's geometric description of compact manifolds with Killing
spinors \cite{baer93a}. Indeed, Gallot's theorem, which says that if $(N,g)$ is compact then the Riemannian cone $(\bar N, \bar g)$ is irreducible or flat, can be restated as follows: the locally metric connection $D$ on the compact manifold $M:=S^1\times N$ defined in Example \ref{ex1} is either flat or irreducible (note that, by construction, $D$ is not globally metric).

One of the main results of this paper is the following:
\begin{ath}\label{tm1}
Conjecture \ref{conj} holds provided that $D$ is {\em tame} ({\em cf.}
Definition \ref{tamme}).  \end{ath}

The tameness condition is related to the life-time of
incomplete geodesics of the connection. It is satisfied by the
connections covered by Gallot's theorem, by every connection
$C^1$-close to them ({\em cf.} Theorem \ref{geod} and Remark \ref{et}), and by large classes of locally metric connections defined on generalized cylinders ({\em cf.} Example \ref{gc}).  

We will now give a more detailed account of our results from the point
of view of conformal geometry. On a conformal manifold $(M,c)$, the
r\^ole of the Levi-Civita connection is played by the family of {\em Weyl
connections}, which are torsion-free connections on the
tangent bundle 
$TM$ preserving the conformal structure \cite{weyl}. 
Weyl connections can be either {\em closed}, or {\em exact}, 
({\em i.e.} locally, resp. globally equal to the Levi-Civita
connection of some Riemannian metric in the conformal class), or {\em
  non-closed}. In this framework, a locally metric connection
compatible with a conformal structure $c$ is nothing but a closed Weyl
connection on $(M,c)$.

As a consequence of the Merkulov-Schwachh\"ofer classification of 
groups occurring as holonomy of torsion-free connections
\cite{schwach}, the holonomy group of every non-closed {\em
irreducible} Weyl connection is the full
conformal group in dimensions other than 4. In \cite{af}
we show that the reducible case is very interesting and, so far,
little understood: The holonomy reduction defines locally a {\em
conformal product} structure, and the holonomy group, although included in a
product group, is not necessarily a product itself. 
In short, the restricted holonomy
of a non-closed Weyl connection is either trivial, the full conformal
group, some special groups in dimension 4, or it is reducible (in
which case no complete description exists yet).

In contrast to that, the restricted holonomy of a closed Weyl connection 
is always a Riemannian holonomy (see Remark \ref{cor} below). 
However, not every Riemannian holonomy group occurs as holonomy of 
a closed, non-exact Weyl connection. More precisely, we show in
Section 6.1 that the locally symmetric case and the quaternion
K\"ahler holonomy $\Sp_k\cdot\Sp_1$ do not occur, while all other
irreducible holonomy groups in the Berger list occur as restricted
holonomy groups of closed, non-exact Weyl connections. Moreover, one
can even realize them on compact manifolds, by means of the cone
construction, {\em cf.} Theorem \ref{holo} for details. 

Theorem \ref{tm1} excludes (under the tameness assumption), the
existence of non-trivial reducible holonomy groups of closed, non-exact,
Weyl connection on a compact conformal manifold.
Of course, an analogous statement can not hold for exact 
or non-closed Weyl connections. Simple counter-examples are Riemannian
products for the first case and conformal products (see
\cite{af}) for the second one.
\medskip

We now describe the strategy of the proof of the main results. To every closed,
non-exact Weyl connection $D$ on a conformal manifold $(M,c)$, we associate
its {\em minimal Riemannian cover} $(\mm,g_0)$, with the property that
the deck transformation group acts on $\mm$ by strict
homotheties, and 
the pull-back of $D$ to $\mm$ is the Levi-Civita connection of $g_0$.
We obtain in this way a one-to-one correspondence between closed,
non-exact, Weyl connections on compact conformal manifolds and
incomplete Riemannian manifolds carrying a co-compact group $\G$ of
strict homotheties acting freely and properly discontinuously, called 
{\em cone-like} manifolds (see also Remark \ref{cor} for an equivalent definition). 
Every Riemannian cone over a compact manifold is a cone-like manifold (see Example \ref{leit}). 
However, while the group of strict homotheties is at least 1-dimensional on every a Riemannian cone, it is only a discrete group on general cone-like manifolds.

%Despite their generality, cone-like spaces still have some common features with Riemannian cones over complete manifolds:
Our first result which is also of independent interest, gives an information 
(essential for the proof of Theorem \ref{tm1}) about the metric space structure of cone-like manifolds. 
\begin{ath}\label{t1}
Let $(\mm,g_0,\G)$ be a cone-like space and let $d$ denote the distance on $M$ induced by $g_0$. Then
the metric completion of $(\mm,d)$ is a metric
space $\mmm$ such that $\mmm\sm\mm$ is a single point $\o$, called
the {\em singularity} of $\mm$.
\end{ath}

This result will be proved in Section 2 using topological group theory. The
crucial point of the proof is Lemma \ref{2.4}, which states that on a
cone-like manifold, the distance from a fixed point to its image
through any contracting homothety in $\G$ is bounded (a fact which 
does not necessarily hold on the universal covering of $\mm$). 

Theorem \ref{tm1} can now be restated as follows:
If the restricted Riemannian holonomy of a tame cone-like manifold 
is reducible, then the metric is flat. 

The key point of the proof is to show the existence of families of
{\em complete} integral leaves of any of the two integrable foliations
(corresponding to the 
parallel splitting of the tangent bundle), all isometric to each other. 
On the other hand, we show that the homotheties of $(\mm,g_0)$ preserve
these families, and we end up with pairs of complete Riemannian manifolds
which are at the same time isometric {\em and} homothetic to each
other, thus flat.

Roughly, these ideas are inspired by the original proof of Gallot's theorem
\cite{gal}. However, in our more general {\em cone-like} setting, the difficulty
comes from
the lack of information about the incomplete geodesics (which, for a
cone, are just its {\em rays}, the orbits
of the homothety flow). This is where we use the assumption that $D$ is a {\em tame} connection, 
which is equivalent to the existence of
uniform bounds for the life-times of the incomplete geodesics generated by
vectors belonging to any compact subset of the tangent bundle, and allows us to
construct the families of complete submanifolds mentioned above.

Theorem \ref{tm1} applies to a wide class of Weyl connections: We show
in Section 5 that the tameness condition is fulfilled by any small
deformations of a cone metric, and more generally by any Weyl connection $D$ which is 
{\em stable} with respect to a complete metric $g$ in the conformal class:

\begin{ath}\label{geod} A Weyl connection on a conformal manifold $(M,c)$
which is stable with respect to some complete metric $g\in c$ is tame.
\end{ath}

The stability condition ({\em cf.} Definition \ref{at}) is equivalent to a
system of differential inequalities (Equations \eqref{s1}-\eqref{s2} below) and is therefore an
open condition. It is satisfied by a $C^1$-neighborhood of
the canonical Weyl connection on the quotient of a cone by one of its
homotheties. Another class of stable Weyl connections is described in Example \ref{gc}.

% The proof of Theorem \ref{geod} relies on a careful
% analysis of the ``slope'' and ``speed'' of the geodesics of a $g$-stable Weyl connection, which by
% the stability condition satisfy a system of differential inequalities (Equations 
% \eqref{s1}-\eqref{s2} below). A qualitative study of this system (Lemmas \ref{z}-\ref{est}) gives
% the desired bounds for the life-time of incomplete geodesics.

We classify, at last, all possible restricted holonomy
groups of closed tame Weyl connections in Theorem \ref{holo}. For some
of these groups, the tameness condition turns out to be automatic. We
give, moreover, a full geometrical description of the underlying
manifolds in Theorem \ref{struct}. 

{\sc Acknowledgment.} We thank the anonymous referee for having pointed out an error in a previous version of this article and for several suggestions which significantly improved the exposition. 

\vs

\section{The minimal Riemannian cover of a closed Weyl connection}\label{2}

In this section, $(M,c)$ denotes a connected conformal manifold
and $D$ denotes a
closed, non-exact, Weyl connection on $(M,c)$ (see {\em e.g.} \cite {af} for 
the basic definitions). Let
$\pi:\tilde M\to M$ be the universal cover of $M$, endowed with
the induced 
conformal structure $\tilde c:=\pi ^*c$, and Weyl derivative
$\tilde D:=\pi ^*D$. Since $\tilde M$ is simply connected, $\tilde D$
is exact, so $\tilde M$ carries a Riemannian metric $\tilde
g_0\in\tilde c$, unique up to a multiplicative constant, whose
Levi-Civita covariant derivative is just $\tilde D$.

\begin{elem}\label{l21}
The group $\A\simeq\pi_1(M)$ of deck transformations of the covering
$\tilde M\to M$ consists of homotheties of $\tilde g_0$. 
\end{elem}

\begin{proof}
Every element $\a\in\A$ is a conformal transformation of $(\tilde
M,\tilde c)$, so there exists a positive function $\rho$ such that $\a
^*\tilde g_0= \rho ^2\tilde g_0$. On the other hand, $\a$
 preserves $\tilde D$, so the Riemannian metric $\a
^*\tilde g_0$ is $\tilde D$-parallel, therefore $\rho$ is constant.
\r

For every $\a\in\A$ we denote by $\rho(\a)$ the constant
of homothety. Consider the sub-group of isometric deck transformations of
$(\tilde M,\tilde g_0)$:
$$\I:=\{\a\in\A\ |\ \rho(\a)=1\}.$$
Of course, $\rho$ being a group homomorphism from $(\A,\circ)$ to $(\RM
^*_+,\times)$, $\I$ is a normal subgroup of $\A$. The quotient
manifold $\mm:=\tilde M/\I$ is a Galois covering of $M$ with Abelian deck
transformation group $\G:=\A/\I$, isomorphic to the subgroup
$\rho(\A)$ of $(\RM
^*_+,\times)$. Moreover $\tilde g_0$ projects to a Riemannian metric $g_0$
on $\mm$. Clearly $\rho$ descends to a group homomorphism, also denoted by
$\rho:\G\to\RM ^*_+$, such that $f^*g_0=\rho(f)^2g_0$ for every $f\in
\G$. The pull-back of $D$ to
$\mm$ (still denoted by $D$) is the Levi-Civita connection of $g_0$,
and the deck transformation group $\G$ acts by {\em pure} homotheties on
$(\mm,g_0)$ ({\em i.e.} the only isometry in $\G$ is the identity).
This motivates the following:

\begin{ede}
Let  $D$ be a closed Weyl
connection on a connected conformal manifold $(M,c)$. The triple
$(\mm,g_0,\G)$, together with the
covering $\pi:\mm\to M=\mm/\G$ is called the {\em minimal Riemannian
cover} of $(M,c,D)$.
\end{ede}

Notice that there is no canonical way to choose $g_0$ in its
homothety class, but all the properties we will consider in the sequel
will not depend on such a choice. 

If $d$ denotes the  geodesic distance on $\mm$ induced by the
Riemannian metric $g_0$, every $f\in \G$ is a homothety of the metric
space $(\mm,d)$, {\em i.e.} $d(f(x),f(y))=\rho(f)d(x,y)$
for each $x,\ y\in\mm$.

\begin{ede} \label{cl} A {\em cone-like} space is a locally compact metric space 
$(\mm,d)$ together with a finitely generated, non-trivial 
group $\G$ acting freely and properly discontinuously by homotheties on $(\mm,d)$,
such that $\G$ contains no isometry besides the
identity, and such that the quotient $\mm/\G$ is a compact topological space.
\end{ede}

\begin{ere}\label{cor} 
The above considerations show that the minimal Riemannian cover 
defines a one-to-one correspondence between the set of triples
$(M,c,D)$ consisting in a compact manifold $M$, a conformal structure
$c$ and a closed, non-exact Weyl connection $D$ on it, and the set of
cone-like Riemannian 
manifolds $(\mm,g_0,\G)$ (modulo constant rescalings of the metric $g_0$).
\end{ere}

A fundamental example of cone-like space, which is the {\em
Leitfaden} of our present study, is the following:

\begin{ex}\label{leit} 
Let $(N,g^N)$ be a complete Riemannian manifold and let 
$$(M_0,g_0):=(\R^*_+\times N, dt^2+t^2g^N)$$
be the {\em Riemannian cone} over $N$ (note that $g_0$ and the product metric $g$ on $M_0\simeq \R\times
N$ are conformally related by setting $t=e^s,\ t\in\R^*_+,\ s\in\R$). The multiplication by 
some $\l>1$ on the $\R$-factor is a strict homothety of $g_0$ and an isometry of $g$. It generates a group $\G$ 
acting freely and properly discontinuously on $\mm$. The metric $g$ projects to the product metric, also denoted by $g$, on 
the quotient manifold $M:=\mm/\G\simeq S^1\times N$. The Levi-Civita connection $D_0$ of $g_0$ is
$\G$-invariant, inducing therefore a closed, non-exact Weyl connection $D$ on $(M,[g])$. It is
straightforward to check that
$(\mm,g_0)$ is the minimal Riemannian cover of $(M,[g],D)$. The theorem of Gallot \cite{gal} about the irreducibility
of the Riemannian cone $(\R^*_+\times N, dt^2+t^2g^N)$ is equivalent to the fact that $D$ is either flat or irreducible.
\end{ex}

Metric cones can be equivalently characterized by the existence of a global {\em homothetic gradient
flow}, {\em i.e.} a complete vector field which is locally a gradient (with respect to a local
$D$-parallel metric $g_0$), and acts infinitesimally by homotheties of $g_0$. We will exhibit in
this section some further properties which the class of cone-like spaces shares with the (much more
restricted) class of Riemannian cones. The most important such property is
that a cone-like space still has an ``apex'', more precisely, it can
be completed by adding one point. This is exactly the statement of Theorem \ref{t1}, which we will now prove.

%The terminology in Definition \ref{cl} is justified by Theorem \ref{t1}, which says that 
%$(\mm,d)$ can be completed by adding one single point. 

\begin{proof}[Proof of Theorem \ref{t1}]
Since every commutator of $\G$ is an isometry of $(\mm,d)$, the hypothesis ensures
that $\G$ is Abelian.
We need to show that $(\mm,d)$ contains at least one non-convergent Cauchy
sequence, and that any two such sequences are equivalent. 

Let $f\in\G$ be any element with $\rho(f)<1$. For every $x\in\mm$
and $m<n\in\N$ we
have 
$$d(f^m(x),f^n(x))\le\sum_{k=m}^{n-1}d(f^k(x),f^{k+1}(x))=
d(x,f(x))\sum_{k=m}^{n-1}\rho(f)^k<
d(x,f(x))\frac{\rho(f)^m}{1-\rho(f)},$$ 
thus showing that $\{f^n(x)\}$ is a Cauchy sequence. If this sequence
had a limit $l$ in $\mm$, then $l$ would be a fixed point of $f$,
contradicting the fact that $\G$ acts freely. Thus $(\mm,d)$ is non-complete.

\begin{elem}\label{2.3}
Let $\{x_n\}$ be a non-convergent Cauchy sequence in $(\mm,d)$. Then
there exists $x\in \mm$ and a sequence $\{f_n\}$ of elements of $\G$
satisfying $\lim_{n\to\infty}\rho(f_n)=0$ such that $\{x_n\}$ is
equivalent to $\{f_n(x)\}$.
\end{elem}

\begin{proof}
Let $\pi$ denote the projection of $\mm$ onto the compact space
$M:=\mm/\G$. By choosing a subsequence if necessary, we may assume that
$\pi(x_n)$ converges to some $y\in M$. Take $x\in\pi ^{-1}(y)$. Since
$\G$ acts properly discontinuously, there exists some open
neighborhood $U_0$ of $x$ such that $h(U_0)\cap U_0=\emptyset$ for
every $h\in\G$ different from the identity. We choose $r>0$ such that
the ball $B_x(2r)$ of radius $2r$ in $x$ lies in $U_0$. 
Then $U:=\pi(B_x(r))$ is a
neighborhood of $y$ in the quotient topology, so there exists some
$n_0$ such that $\pi(x_n)\in U$ for $n\ge n_0$. This shows that for
$n\ge n_0$ there exist $z_n\in B_x(r)$ and $f_n\in \G$ such that
$x_n=f_n(z_n)$. 

Suppose that $\rho(f_n)$ does not tend to zero. By taking a
subsequence if necessary, we may assume that $\rho(f_{n})>\delta$ for
every $n$. For every $m,n$ such that $f_n\ne f_m$, the open balls
$f_{n}(B_x(2r))$ and 
$f_{m}(B_x(2r))$ are disjoint, being included in $f_n(U_0)$,
and $f_m(U_0)$ respectively. As $f_{n}(z_n)\in
f_{n}(B_x(r))$ and $f_{m}(z_m)\in
f_{m}(B_x(r))$, we get $d(x_n,x_m)=d(f_{n}(z_n),f_{m}(z_m))\ge
2r\delta$. The fact that
$\{x_n\}$ is a
Cauchy sequence ensures therefore the existence of an index $N$ such that
$f_{n}=f_{N}$ for every 
$n>N$. Since $B_x(r)$ is relatively compact, we may assume (passing
to some subsequence, if necessary) that $z_{n}$ tends to $z\in
\overline{B_x(r)}$. Thus $\{x_{n}\}$ converges to $f_{N}(z)$,
contradicting the fact that $\{x_n\}$ does not converge.

This shows that $\lim_{n\to\infty}\rho(f_n)=0$. Since
$d(f_n(z_n),f_n(x))=\rho(f_n) d(z_n,x)<r\rho(f_n)$, the sequences
$\{x_n\}$ and $\{f_n(x)\}$ are equivalent, thus proving the lemma.
\r

In order to conclude the proof of the theorem we need one more
technical result.

\begin{elem}\label{2.4}
For every fixed point $x\in\mm$ there exists a constant $K_x$, depending
on $x$, such that $d(x,f(x))<K_x$ for every contracting $f\in\G$
({\em i.e.} with $\rho(f)<1$). 
\end{elem}

\begin{proof}
Let $\{h_1,\ldots,h_n\}$ be a system of generators of $\G$ with
$\rho_i:=\rho(h_i)>1$ and let $$D_x:=\max_{\{i=1,\ldots,n\}}d(x,h_i(x)).$$ For
every $(a_1,\ldots,a_k)\in \N^k$, we claim that 
\beq\label{in}d\bigg(x,\bigg(\prod_{i=1}^kh_i ^{a_i}\bigg)(x)\bigg)\le
D_x\prod_{i=1}^k\frac{\rho_i ^{a_i+1}-1}{\rho_i-1}.\eeq 
We prove the claim by induction on $k$. For $k=1$ we have 
$$d(x,h_1^{a_1}(x))\le\sum_{s=0}^{a_1-1}d(h_1^s(x),h_1^{s+1}(x))=
d(x,h_1(x))\sum_{s=0}^{a_1-1}\rho_1^s\le
D_x\frac{\rho_1^{a_1}-1}{\rho_1-1} <D_x\frac{\rho_1^{a_1+1}-1}{\rho_1-1}.$$ 
Assume now that \eqref{in} holds for each $k\le l$ and for every
$(a_1,\ldots,a_k)\in \N^k$ and consider some element
$(a_1,\ldots,a_{l+1})\in \N^{l+1}$. We denote by 
$$h:=\prod_{i=1}^lh_i ^{a_i}\ \mbox{ and by }\
y_j:=\left(h_{l+1}^j\circ h \right)(x),\ \forall\; j=0,\dots,a_{l+1}.$$
Using \eqref{in} for $k=l$ we have
$$d(x,y_0)\le D_x\prod_{i=1}^k\frac{\rho_i
  ^{a_i+1}-1}{\rho_i-1},$$
and $d(y_j,y_{j+1})=\rho_{l+1}^j d(y_0,y_1)=\rho(h)\rho_{l+1}^j
d(x,h_{l+1}(x))$, which further imply
\bea
d\bigg(x,\prod_{i=1}^{l+1}h_i ^{a_i}(x)\bigg)&=&d(x,y_{a_{l+1}})\le
d(x,y_0)+\sum_{j=0}^{a_{l+1}-1}d(y_j,y_{j+1}) \\
&\le& D_x\prod_{i=1}^l\frac{\rho_i
  ^{a_i+1}-1}{\rho_i-1}+D_x\prod_{i=1}^l\rho_i
^{a_i}\sum_{j=0}^{a_{l+1}-1}\rho_{l+1}^j\\
&\le&D_x\prod_{i=1}^l\frac{\rho_i
  ^{a_i+1}-1}{\rho_i-1}\bigg(1+\sum_{j=0}^{a_{l+1}-1}\rho_{l+1}^j\bigg)
\le D_x\prod_{i=1}^l\frac{\rho_i
  ^{a_i+1}-1}{\rho_i-1}\bigg(\sum_{j=0}^{a_{l+1}}\rho_{l+1}^j\bigg)\\
&=&D_x\prod_{i=1}^{l+1}\frac{\rho_i
  ^{a_i+1}-1}{\rho_i-1},
\eea
thus proving our claim for $k=l+1$. In order to finish the proof of
the lemma, let $f\in\G$ be an element with $\rho(f)<1$. By
reordering the system of generators if necessary, we can write 
$$f=\prod_{i=1}^nh_i ^{a_i},\quad \hbox{with } a_i\ge 0\ \hbox{for}\
i\le m\ \hbox{and}\ a_i\le 0\ \hbox{for}\
i\ge m+1.$$  
We denote $b_i:=-a_i\ge 0$ for $i\ge m+1$. 
Using \eqref{in} we obtain
\bea
d(x,f(x))&=&\bigg(\prod_{i=m+1}^n\rho_i ^{a_i}\bigg)d\bigg(\prod_{i=1}^mh_i
^{a_i}(x), \prod_{i=m+1}^nh_i ^{b_i}(x)\bigg)\\
&\le&\bigg(\prod_{i=m+1}^n\rho_i ^{a_i}\bigg)\bigg(d\big(x,\prod_{i=1}^mh_i
^{a_i}(x)\big)+d\big(x, \prod_{i=m+1}^n h_i ^{b_i}(x)\big)\bigg)\\
&\le&\bigg(\prod_{i=m+1}^n\rho_i ^{a_i}\bigg)\bigg(D_x\prod_{i=1}^m\frac{\rho_i
  ^{a_i+1}-1}{\rho_i-1}+D_x\prod_{i=m+1}^n\frac{\rho_i
  ^{b_i+1}-1}{\rho_i-1}\bigg).
\eea
We neglect the $-1$ terms in the numerators above and multiply
the brackets. Remembering that $\prod_{i=1}^n\rho_i^{a_i}=\rho(f)<1$,
we finally get 
\bea 
d(x,f(x))&\le&
D_x\bigg(\prod_{i=1}^m\frac{\rho_i}{\rho_i-1}\prod_{i=1}^n\rho_i^{a_i}+ 
\prod_{i=m+1}^n\frac{\rho_i}{\rho_i-1}\bigg)\\  
&\le&D_x\bigg(\prod_{i=1}^m\frac{\rho_i}{\rho_i-1}+
\prod_{i=m+1}^n\frac{\rho_i}{\rho_i-1}\bigg)\\
&\le&D_x\bigg(\prod_{i=1}^n\frac{\rho_i}{\rho_i-1}+1\bigg)=:K_x,
\eea
where the last inequality follows from the fact that $a+b\le ab+1$ for
all $a,b\ge 1$. 
\r

Let now $\{x_n\}$ be a non-convergent Cauchy sequence in
$\mm$. Choose $y\in\mm$ and $f\in\G$ such that $\rho:=\rho(f)<1$. We
claim that $\{x_n\}$ is equivalent to $\{f^n(y)\}$.
By Lemma \ref{2.3}, there exists $x\in \mm$ and a sequence
$\{f_n\}$ of elements of $\G$ 
satisfying $\lim_{n\to\infty}\rho(f_n)=0$, such that $\{x_n\}$ is
equivalent to $\{f_n(x)\}$.  Since
$\lim_{n\to\infty}\rho(f_n)=0$, there exists an increasing sequence of integers
$\{k_n\}$ such that $\rho(f_{k_n})<\rho ^{n}$. As $\rho(f^{-n}\circ
f_{k_n})<1$, Lemma \ref{2.4} yields
$$d(f^n(x),f_{k_n}(x))=\rho ^n d(x,(f^{-n}\circ f_{k_n})(x))\le K_x\rho
^n.$$
The sequences $\{f_{k_n}(x)\}$ and $\{f^n(x)\}$ are thus equivalent,
so the same holds for $\{x_n\}$ and $\{f^n(x)\}$. Finally, for any $y\ne x$,
$\{f^n(x)\}$ is clearly equivalent to $\{f^n(y)\}$, thus finishing the
proof of the theorem.
\end{proof}

Theorem \ref{t1} shows that, despite the fact that cone-like spaces
only carry a discrete group of homotheties, they still have the
one-point completion like all Riemannian cones. Note that the universal
covering of a Riemannian cone is a Riemannian cone itself, therefore admits
the one-point completion as well. It is unknown whether this fact
holds for the universal covering of an arbitrary cone-like space (see
Section \ref{55}). 

\vs

Functions on a Riemannian cone measuring geometric quantities 
like lengths, are equivariant with respect to the radial flow (acting by
homotheties), and thus vary linearly on the rays.

In the more general case of cone-like spaces, we introduce, for
further use, the following simple notion:

\begin{ede}\label{q-lin}
Two positive functions $f_1,f_2:\mm\ra\R_+^*$ are said to be {\em
  equivalent} if 
their ratio is bounded above and below by positive constants. A
function which is equivalent to the distance to the singularity $\o$ is
called {\em quasi-linear}.
\end{ede}

Denote by $\dt:\mm\ra\R^*_+$ the distance to the singularity $\o\in\mmm$:
$\dt(x):=d(x,\o)$. 

\begin{elem} \label{l3.5} Let $\psi:\mm\to\R_+^*$ be any
  $\G$-equivariant function {\em of weight 1} on $\mm$ ({\em i.e.}  
satisfying ${\psi\circ f}=\rho(f){\psi}$ for every element $f\in\G$), such
that $\psi$ and $\tfrac1\psi$ are locally bounded 
({\em e.g.}, $\psi$ is continuous).  Then $\psi$ is quasi-linear.
\end{elem}
\begin{proof}
Consider a compact fundamental domain $\Omega$ of
the action of $\G$ on $\mm$ and define
$$k_1:=\inf_{x\in \Omega}\frac{{\psi(x)}}{\dt(x)},\qquad k_2:=\sup_{x\in
  \Omega}\frac{ {\psi(x)}}{\dt(x)}.$$ 
Because $\dt$ is continuous and $\psi$ and its inverse are locally
bounded, their quotients $\dt/\psi$ and $\psi/\dt$ are bounded on the 
compact set $\Omega$. It follows that $k_1,k_2$ are positive real
numbers, so that
$$\frac{{\psi(x)}}{\dt(x)}\in[k_1,k_2]$$ 
holds tautologically on $\Omega$. Let now $y$ be an arbitrary point of
$\mm$ and $f\in\G$ such that $x:=f^{-1}(y)\in \Omega$.  From the
equivariance property of $\psi$ we get
$$\frac{ {\psi(y)}}{\dt(y)}=\frac{ {\psi(f(x))}}{\dt(f(x))}=
\frac{\rho(f) {\psi(x)}}{\rho(f)\dt(x)}=\frac{
{\psi(x)}}{\dt(x)}\in[k_1,k_2],$$ 
which finishes the proof.
\r

As a consequence of the previous lemma, we show for later use that 
if $(\mm,g_0)$ is the minimal Riemannian cover of a closed non-exact
Weyl connection $D$ on a compact conformal manifold $(M,c)$, then any 
conformal factor
relating $g_0$ to the pull-back on $\mm$ of a metric in the conformal 
class $c$ on $M$ is equivalent to the distance function $\dt$ to the 
singularity $\omega\in\mm$:

\begin{elem}\label{3.5} 
Let $g$ be the pull-back to $\mm$ of a metric in $c$ on $M$ and 
let $\f:\mm\to\RM^*_+$ be defined by 
$g_0=\f^2 g$. The function $\f$ is then quasi-linear
on $\mm$.
\end{elem}
\begin{proof}
Every element $f\in\G$ being an
isometry of $g$, we obtain
$$\rho(f)^2{\f^2}g=\rho(f)^2g_0=f^*g_0={(\f\circ f)^2}g,$$
showing that ${\f\circ f}=\rho(f){\f}$. The assertion thus follows from
Lemma \ref{l3.5}. 
\r

\vs

\section{Tame connections and their geodesics}\label{3}

In contrast to the Riemannian situation, a Weyl connection on a compact
conformal manifold is not necessarily geodesically complete.

\begin{ex}
Let $(M,c)$ be a compact conformal manifold and let $D$ be a closed,
non-exact, Weyl connection on $M$. Theorem \ref{t1} shows that the
minimal Riemannian cover $(\mm,g)$ of $(M,c,D)$ is incomplete, so
through every point of $\mm$ passes an incomplete geodesic. Its
projection onto $M$ is thus an incomplete geodesic of $D$.
\end{ex}

In order to study the geometry of $\mm$ in the neighborhood of its
singularity $\o$, we need to understand the behavior of the geodesics
passing through or near $\o$. In principle, the dynamics of the
geodesic flow of $(M,g)$ can be rather wild near $\o$. Here is a list
of phenomena which may occur: \bi
\item The lengths of the geodesics starting at some given point $P$ and
  passing through $\o$ ({\em i.e.} the life-time of an incomplete geodesic) 
  might not be bounded.
\item There might exist closed geodesics through $\o$ ({\em i.e.} geodesics
  having finite life-time in both directions).
\item There might even exist a complete geodesic whose adherence
contains $\o$.  \ei

\subsection{Tame connections }
To begin with, let us recall some basic facts about the geodesic flow of 
an affine connection $D$ on a manifold $M$, or, equivalently, the exponential 
map
$$\exp^D:\U\ra M.$$
defined on a (maximal) open subset $\U$ of $TM$.
For $X\in\U_x:=\U\cap T_xM$, $\exp^D(X)$ is the
point $\c(1)$ on the geodesic defined by 
$$\c(0)=x\ \mbox{and}\ \dot\c(0)=X.$$ 
We define the {\em life-time} $\ll^D:TM\ra(0,+\infty]$ of a
half-geodesic generated by $X\in
TM$, by
$$\ll^D(X):=\sup\{t>0\ |\ tX\in\U\},$$ 
in other words, the supremum
of the time for which the half-geodesic tangent to $X$ is defined. 
Of course, if $(M,D)$ is geodesically complete, all
life-times are infinite. 

We split the complement $TM\sm\{0\}$ of the zero section in the tangent bundle 
into two sets, the set $\I^D$ of vectors
generating incomplete half-geodesics, and its complement $\mathcal{C}^D$. 
These subsets are both star-shaped, {\em i.e.} for a vector $X\in TM$
$$X\in\I^D\ \iff sX\in\I^D,\qquad \forall\ s>0.$$ 
The two sets $\I^D$ and $\mathcal{C}^D$ are in
general neither open, nor closed. They are however Borel measurable, as $\mathcal{C}^D$ is an
infinite intersection of open sets.

For the following definition, we need to consider a Riemannian metric $g$
on $M$, in order to define the sphere bundle $S^gM$ of unit vectors
on $M$. However, the notion that we are about to define does not depend
on the choice of such a metric:

\begin{defi}\label{tamme} A connection $D$ on $M$ is {\em tame} if and
  only if the function
$$\mu^g:\mm\ra [0,+\infty],\qquad \mu^g(x):=\sup\left\{\{0\}\cup\{\ll^D(X)\ |\
X\in\I^D_x\cap S^g_xM \}\right\}$$
is locally bounded on $M$. 
\end{defi}
Of course, $\mu^g=0$ if and only if $D$ is geodesically complete.

If $D:=\nb^g$ is the Levi-Civita connection of a Riemannian manifold $(M,g)$, 
$\mu^g(x)$ is the supremum of the lengths of all incomplete
half-geodesics starting in $x$. If $\nb^g$ is tame, we say that
$(M,g)$ is a {\em tame} Riemannian manifold. Note that, while the
tameness of the connection $D$ does not depend on the auxiliary metric
in Definition \ref{tamme}, the tameness of a Levi-Civita connection
depends on the corresponding metric (see Example \ref{exf} for an
example of a non-tame Riemannian manifold, which
becomes complete after a conformal rescaling).

\subsection{Tame cone-like spaces }
If $D$ is a connection on $M$, the induced connection (still denoted
by $D$) on some covering $M_0$ of $M$ is tame if and only if $D$ is tame on $M$.
A tame closed Weyl connection $D$ on a compact manifold $(M,c)$ is
therefore equivalent to a tame cone-like Riemannian manifold
$(M_0,g_0)$. Here we have the following criterion:

\begin{epr}\label{tam} 
A cone-like manifold $(\mm,g_0)$ is tame if and only if $\mu:\mm\ra
(0,+\infty]$ is (finite and) quasi-linear on $\mm$, i.e. if there
exists a constant $K>0$ such that
\be\label{tamql}\delta(x)\le\mu(x)\le K\delta(x),\ \forall x\in M_0,\ee
where $\dt:M_0\ra\R^*_+$ is the distance from a point to the
singularity $\omega$.
\end{epr}
\begin{proof} The distance $\dt$ to the singularity $\o\in \mm$ 
is always continuous on $\mm$, and $\mu\ge
  \dt$, therefore $\tfrac1\mu$ is locally bounded. On the other hand,
  $\dt$ and $\mu$ are clearly $\G$-equivariant of weight 1, because
  both denote geometrical lengths. Therefore, if (\ref{tamql}) holds
  on a fundamental domain of $\G$, then it holds on whole $\mm$. As
  such a fundamental domain is relatively compact, the quasi-linearity
  of $\mu$ is thus equivalent to its local boundedness. 
% If $D$ is tame, Lemma \ref{l3.5}
%   implies, together with Lemma \ref{tr}, the quasi-linearity of $\mu$.
% Conversely, if $\mu$ is quasi-linear, then it is locally
% bounded, therefore, again by Lemma \ref{tr}, $D$ is tame.
\end{proof}

It is thus obvious that a cone over a complete Riemannian manifold is
tame (its only incomplete geodesics are its rays, thus $\mu=\dt$). On
the other hand, not all cone-like manifolds are tame, as we will see
in Section \ref{ntame}, therefore not all closed Weyl connections on a
compact manifold are tame. We will however show in 
 that the
tameness condition holds for an open set of Weyl connections (in
the $C^1$-topology) containing the ones constructed (as in Example
\ref{leit}) from Riemannian cones.

\vs
 
\section{Closed Weyl connections with reducible holonomy}\label{4}

The goal of this section is to prove our main result, Theorem \ref{tm1}, concerning the holonomy
of tame locally metric connections. We will rephrase it in terms of Weyl connections.

\begin{ath}\label{main}
If the restricted holonomy representation of a closed, non-exact, tame
Weyl connection $D$ on a compact conformal manifold $(M,c)$ is
reducible, then $D$ is flat.
\end{ath}

\begin{proof} We start by showing that if the restricted holonomy $\Hol_0(D)$ 
is reducible, then there exists a finite covering $\bar M$ of $M$ on
which the full holonomy of the pull-back of $D$ has reducible
holonomy. In order to keep the argument as simple as possible, we will
not be very precise on the holonomy groups and consider them as
abstract groups rather than as transformation groups of each tangent
space.

Consider the metric $\t g_0$ on the universal cover $\t M$ of $M$
(defined up to a multiplicative constant), whose Levi-Civita covariant
derivative $\t\n$ is the pull-back of $D$ to $\t M$. The holonomy of
$\t\n$ is clearly equal to the restricted holonomy of $D$. By Theorem
IV.5.4 in \cite{kn1}, the tangent bundle of $\t M$ splits in a direct sum
$T\t M=T_0\oplus\ldots \oplus T_m$ of $\t\n$-parallel sub-bundles and
the holonomy group of $\t\n$ satisfies $\Hol(\t\n)=H_1\times \ldots\times
H_m$, where $H_i$ acts irreducibly on $T_i$ and trivially on $T_j$ for
$j\ne i$ ($T_0$ being the flat component).  Moreover this
decomposition is unique up to a permutation of the set
$\{1,\ldots,m\}$ (such permutations may occur if some of the factors
$H_i$ coincide).

By Lemma \ref{l21}, every element $f\in\A$ of the deck transformation
group of the covering $\t M\to M$ is affine with respect to $\t\n$, so
there exists a permutation $\s_f$ of $\{1,\ldots,m\}$ such that
$f_*(T_i)=T_{\s_f(i)}$. Let $\B\subset \A$ be the kernel of the group
homomorphism $\A\to S_m$ given by $f\mapsto \s_f$. The metric $\t g_0$ and
the connection $\t\n$ on $\t M$ descend to a conformal structure $\bar
c$ and a Weyl connection $\bar D$ on the quotient $\bar M:=M/\B$, which
is a finite covering of $M$ with group $\A/\B\subset S_m$. By
construction, the holonomy group of $\bar D$ on $\bar M$ is reducible.

Replacing $(M,c,D)$ by $(\bar M,\bar c,\bar D)$, we can from now on
assume that the full holonomy of $D$ is reducible. This implies that
the tangent bundle of the minimal Riemannian cover $(\mm,g_0)$ of
$(M,c,D)$ splits in a direct sum of orthogonal distributions
$TM=V_1\oplus V_2$, parallel with respect to the Levi-Civita
connection $\n^0=D$ of $g_0$. These distributions are integrable,
hence define two orthogonal (and complementary) foliations on $\mm$.

We will use the notion {\em maximal leaf} $M_i$, $i=1,2$, through
$x\in \mm$ to denote the set of points that can be connected to $x$ by
means of a smooth curve tangent to $V_i$. It is a standard fact that $M_i$ have smooth structures such that
$M_i\to M_0$ are immersions (although $M_i$ are not necessarily submanifolds of $\mm$).

We start with two preliminary results which hold on every (not
necessarily complete) reducible Riemannian manifold $(\mm,g_0)$.
The first one is an elementary consequence of the local de Rham decomposition theorem.

\begin{elem} \label{exp-loc}
Let $U_1$ be a local leaf of $V_1$ (connected but not necessarily complete) and assume that $X\in
\Gamma(V_2)$ is a parallel vector field along $U_1$. Assume moreover that $\exp_x tX$
is defined for all $x\in U_1$ and $t\in[0,1]$. Then $x\mapsto
\psi(x):=\exp_x X$ maps $U_1$ locally isometrically onto its image $U'_1$ (which is itself a leaf of $V_1$).
\end{elem}

%\begin{elem} \label{exp-loc}
%Let $U_1$ be a maximal leaf of $V_1$ through $x\in M_0$ and assume that $X\in
%V_2(x)$ is a vector such that for every $y\in U_1$ and for every path $\gamma$ in $U_1$ from $x$ to $y$, $\exp_y tY$
%is defined for all $t\in[0,1]$, where $Y$ denotes the parallel transport of $X$ along $\gamma$. 
%
%Then $x\mapsto
%\psi(x):=\exp_x X$ maps $U_1$ isometrically onto its image.
%\end{elem}

\begin{proof} 
Consider the map $\f:U_1\times
[0,1]\to \mm$ defined by $\f(x,t):=\exp_x tX$. Define $X_{(x,t)}\in
T_{\f(x,t)}\mm$ by
$$X_{(x,t)}:=\frac{d}{ds}\bigg|_{s=t}\f(x,s).$$ In other words,
$X_{(x,t)}$ is the tangent vector to the geodesic $s\to \exp_x sX$ at
$s=t$, so we clearly have the relation
\beq\label{pse}\f(x,t+s)=\exp_{\f(x,t)}sX_{(x,t)}.  \eeq Let us fix
$x\in U_1$ and denote $x_t:=\f(x,t)$. The local de Rham decomposition
theorem (Proposition IV.5.2 in \cite{kn1}) states that each $x_t$
has a neighborhood $U(t)$ in $\mm$ isometric to a Riemannian product
$U(t)\simeq U_1(t)\times U_2(t)$, where $U_1(t)$ and $U_2(t)$ are
local leaves of $V_1$ and $V_2$ through $x_t$.
  
  The geodesic segment $\f(\{x\}\times [0,1])$ is compact, so it can
  be covered by a finite number of neighborhoods $U_2(s_1),\ldots
  ,U_2(s_n)$, with $0=s_1<\ldots<s_n=1$. Choose now $t_i\in
  (s_i,s_{i+1})$ $\forall\; i=1,\dots,n-1$, such that $\f(x,t_i)\in
  U_2(s_i)\cap U_2(s_{i+1})$, and set $t_n:=1$. For $k=1,\ldots,n$,
  let $V_k$ be the open subset of $U_1$ defined by
  $$V_k:=\{y\in U_1\ |\ \f(y,s_k)\in U(s_k)\}.$$ We denote by $V$ the
  intersection of the $V_k$'s and by $W_k$ the subset of $U(t_k)$
  given by $W_k:=\f(V\times\{t_k\})$.\vs

\begin{center}\input{f2.pstex_t}\vs
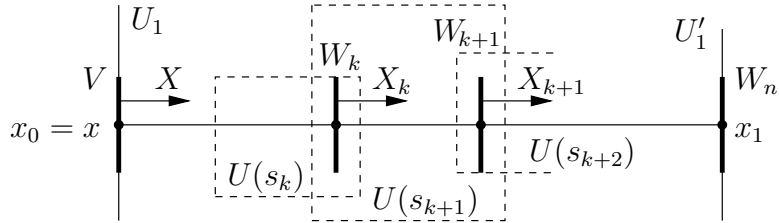
{\sc Figure 1. } Stepwise 
exponentiation along the geodesic segment $\f\left(\{x\}\times
[0,1]\right)$.\end{center}
  
  Consider the vector field $X_k$ along $W_k$ whose value at
  $\f(y,t_k)$ is $X_{(y,t_k)}$.  By construction, there exists a
  bijection $\f_k:W_k\to W_{k+1}$ defined by
  $\f_k(\f(y,t_k)):=\f(y,t_{k+1})$ for all $y\in V$.
  
We claim that for every $k=1,\ldots, n$, \bi
\item $\f_{k-1}$ is an isometry;
\item the vector field $X_{k}$ is parallel along $W_{k}$;
\item $W_{k}$ is an open subset of the leaf $U_1(t_{k})$.  \ei For
$k=1$, the first statement is empty, and the other two hold by
hypothesis. 
  
  Assume that the claim holds for some $k\ge 1$.  Since $X_k$ is
  parallel along $W_k$, it is constant in the product coordinates on
  $W_k\times U_2(s_{k})\subset U(s_{k})$ ({\em i.e.} there exists $Z\in
  T_{x_{t_k}}U_2(s_k)$ such that $(X_k)_{(z,x_{t_k})}=(0,Z)$, $\forall\;
  z\in W_k$).  By \eqref{pse} we have
  $\f_k(\f(y,t_k))=\exp_{\f(y,t_k)}(t_{k+1}-t_k)X_{(y,t_k)}$ so in the
  product coordinates $\f_k(z,x_{t_{k}})=(z,x_{t_{k+1}})$ for all
  $z\in W_k$, showing that $W_{k}\subset U_1(t_{k})$ and that $\f_k$
  is an isometry. Moreover $X_{k+1}$ is constant along $W_{k+1}$ in
  these coordinates, thus proving the induction step.

We have shown that in the neighborhood $V$ of $x$ in $U_1$, the map
$x\mapsto \psi(x)= \exp_x X$ is the composition of $n-1$ isometries
$\f_{n-1}\circ\cdots\circ \f_1$ between local leaves of the
distribution $V_1$. As this holds in the neighborhood of every point
$x$ of $U_1$, $\psi$ is a local isometry from $U_1$ to its image
$U'_1$. In particular, this shows that $U'_1$ is an open subset of the
complete integral leaf of $V_1$ passing through $\f(x,1)$.
%Moreover, the claim above shows that $Y:=\psi_*(X)$ is a well-defined
%parallel vector field along $U'_1$. Consider the map $\t\psi:U'_1\to
%\mm$ defined by $y\mapsto \exp_y(-Y)$. From the local considerations
%above, it is clear that $\t\psi\circ\psi$ is the identity of
%$U_1$. This shows that $\psi$ is one-to-one.  
\r

The next result, which is elementary as well,
shows that exponentiating a geodesic tangent to $V_1$ in the direction
of a constant or affine Jacobi field tangent to $V_2$ yields another
geodesic whenever it is defined.

\begin{elem}\label{geo}
Let $\c:[a,b]\to M_0$ be a geodesic tangent to $V_1$ parametrized by
  arc-length and let $X\in T_{\c(a)}M_0$ be a vector tangent to
  $V_2$. Extend $X$ to a parallel vector field along $\c$.

{\em (i)} Assume that $\c_s(t):=\exp_{\c(t)}(sX)$ is well-defined for
  all $t\in[a,b]$ and $s\in[0,1]$. Then $\c_1(t)$ is a geodesic in
  $M_0$ and its tangent vector at $t$ is the parallel transport of
  $\dot \c(t)$ at $\exp_{\c(t)}(X)$ along the geodesic $s\mapsto
  \exp_{\c(t)}(sX)$.

{\em (ii)} Assume that $\c^X(t):=\exp_{\c(t)}(tX)$ is well-defined for
  all $t\in[a,b]$. Then $\c^X(t)$ is a geodesic in $M_0$ and the
  projections of $\dot \c^X(t)$ onto $V_1$ and $V_2$ are parallel
  vector fields along $\c^X$ of length $1$ and $|X|$ respectively.

\end{elem}

\begin{proof}
(i) The first statement follows immediately from Lemma
\ref{exp-loc}. The second one is a consequence of the inductive claim used to
prove the same lemma.

(ii) Assume first that $M_0$ is a global Riemannian product
$M_0=M_1\times M_2$. If $\c(a)=(m_1,m_2)$, then $\c(t)=(\c_1(t),m_2)$
for some geodesic $\c_1$ in $M_1$ parametrized by arc-length. The
vector field $X$ along $\c$ can be written $X=(0,X_2)$, where $X_2$ is
a constant vector tangent to $M_2$ at $m_2$. Denoting by
$\c_2(t)=\exp_{m_2}tX_2$ the geodesic in $M_2$ starting at $m_2$ with
initial speed $X_2$, then $\c^X(t)=(\c_1(t),\c_2(t))$, is a
geodesic in $M_0$. The projections of $\dot \c^X(t)$ onto $V_1$ and
$V_2$ are $(\dot \c_1,0)$ and $(0,\dot\c_2)$, which are clearly
parallel vector fields along $\c^X$ of length $1$ and $|X|$
respectively.

Back to the general case, it is of course enough to show that the
statement holds in the neighborhood of every point $\c^X(t_0)$.
Since the domain of definition of the exponential on the normal bundle
of a geodesic is open, the curve $c(t):=\exp_{\c(t)}(t_0X)$ is
well-defined for $t$ near $t_0$. By Lemma \ref{exp-loc}, $c(t)$ is a
geodesic through $x:=\c^X(t_0)$, parametrized by arc-length. Moreover,
if $Y$ denotes the parallel vector field along $c(t)$ with
$Y_x=d\exp_{\c(t_0)}(t_0X)$, Lemma \ref{exp-loc} also shows that
$Y_{\c(t)}=d\exp_{\c(t)}(t_0X)$, so by \eqref{pse},
$\c^X(t)=\exp_{\c(t)}((t-t_0)Y)$.  \vs 

\begin{center}\input{f3.pstex_t}\vs
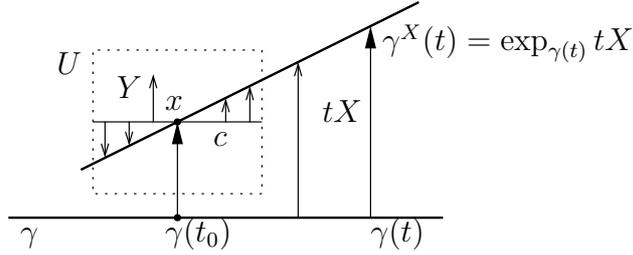
{\sc Figure 2.} Idea of the proof of Lemma \ref{geo} (ii).
\end{center}

By the local de Rham theorem, the
point $x$ has a neighborhood $U$ isometric to $U_1\times U_2$, where
$U_i$ is some local leaf of $V_i$ through $x$. As $\c^X(t)$ lies in
$U$ for $t$ near $t_0$, the statement follows from the first part of
the proof.
\end{proof}

We assume from now on that $D$ is a closed tame Weyl connection on a
compact conformal manifold $(M,c)$ which has reducible holonomy and
that $(\mm,g_0)$ is the minimal Riemannian cover of $(M,c,D)$. We
denote as before by $d$ the distance induced by $g_0$ on $\mm$, by
$\o$ the singularity of $\mm$ and by $\dt$ the distance to the
singularity: $\dt(x):=d(x,\o)$.  Since $D=\n^0$, $g_0$ has reducible
holonomy, so the above results apply to the present setting.  We will
need the following quantitative version of the local de Rham
decomposition theorem for $(\mm,g_0)$.

\begin{elem}\label{polidisc}
There exists a quasi-linear function $\s:\mm\to \RM^*_+$ such that
each point $x\in\mm$ has a neighborhood $U$ and an isometry $F:U\ra
B^1_x(\s(x))\times B^2_x(\s(x))$, with $F(x)=(x,x)$, where $B^i_x(r)$
is the ball of radius $r$ around $x$ in the maximal leaf $M_i$ though
$x$, tangent to the distribution $V_i$.
\end{elem}
\begin{proof}
By Lemma \ref{l3.5}, it is enough to define $\sigma$ on a relatively
compact fundamental domain $K\subset \mm$ of the covering $\mm\to M$,
and to extend it to $\mm$ in a $\G$-equivariant way by
$\sigma(f(x))=\rho(f)\sigma(x)$ for all $f\in \G$ and $x\in K$.

The local de Rham theorem ensures that for every $x\in \mm$ there
exist neighborhoods $U_i(x)$ of $x$ in the maximal leaf $M_i(x)$
though $x$, tangent to the distribution $V_i$, such that $U_1(x)\times
U_2(x)$ is isometric to a neighborhood $U(x)$ of $x$ in $\mm$. Take a
finite number of points $x_i$ such that $\overline K\subset \cup_i
U(x_i)$. Each neighborhood $U_1(x_i)$ and $U_2(x_i)$ contains a
geodesic ball centered in $x_i$ of radius $r_1(x_i)$ and $r_2(x_i)$
respectively.  It is then enough to define $\sigma$ on $K$ to be the
minimum of all these radii.
\r

We now come to a key point of the proof of Theorem \ref{main}, namely
the existence of {\em complete} maximal leaves tangent to the
distributions $V_i$.

\begin{epr}\label{complete}
If $M_1$ is a maximal leaf of $V_1$ which is incomplete, then every
maximal leaf of $V_2$ which intersects $M_1$ is complete.
\end{epr}

\begin{proof}
Since $M_1$ is totally geodesic and incomplete, through every point 
$x\in M_1$ passes a geodesic
$\c:(0,r]\to M_1$ parametrized by arc-length, such that $\c(r)=x$, 
which can not be defined at $t=0$. Since $M_i$ is totally
geodesic in $\mm$, $\c$ is also a geodesic in $\mm$. By Theorem
\ref{t1}, we must have $\lim_{t\to 0}\c(t)=\o$ in $(\mmm,d)$.

 Let $X\in T_xM\cap V_2$ be any
unit normal vector to $M_1$ at $x$, extended as before to a parallel
vector field along $\c$. We claim that the geodesic generated by $X$
in $\mm$ is complete.

The crucial point here is the fact that every point $\c(t)$ is far
enough from the singularity $\o$, in order to ensure that the
exponential function is well-defined in a suitable neighborhood. More
precisely, Proposition \ref{tam} shows that there exists a constant $\k>0$ such that
for every $t\in(0,r]$, the
distance $\dt(\c(t))$ from $\c(t)$ to $\o$ (in $\mm$) is bounded from below by
${\k t}$. Consequently, $\exp_{\c(t)}sX$ is well-defined
for $|s|\le \k{t}$, so by Lemma \ref{geo} (ii), the curve
$\c_1:(0,r]\to \mm$ defined by $\c_1(t):=\exp_{\c(t)}\k tX$ is a
geodesic in $\mm$ with $|\dot\c_1|^2=1+{\k ^{2}}$. Moreover, the limit
in $\mmm$ of $\c_1(t)$ as $t\to 0$ is clearly $\o$. Proposition
\ref{tam} applied this time to the geodesic parametrized by arc-length
$\t\c_1$ defined by
$$\t\c_1(t):=\c_1((1+\k ^{2})^{-1/2}t)$$ yields
$\dt(\t\c_1(t))>\k{t}$, whence
$$\dt(\c_1(t))>(1+\k ^{2})^{1/2}\k{t}>\k{t}.$$ 
Consequently, for every
$t\in (0,r]$, every geodesic defined by a unit vector $Y\in
T_{\c_1(t)}\mm$ is defined at least up to the time $\k t$.  Taking $Y$
to be the speed vector of the geodesic $s\to \exp_{\c(t)}sX$ at $s=\k t$, we obtain
that this geodesic can actually be extended for $s\in[0,2\k t]$, for
any $t\in (0,r]$.  By Lemma \ref{geo} (ii), the curve $\c_2:(0,r]\to
\mm$ defined by $\c_2(t):=\exp_{\c(t)}2\k tX$ is thus a geodesic in
$\mm$ with $|\dot\c_2|^2=1+{4\k ^{2}}$. Again, we check that the
distance from $\c_2(t)$ to the singularity is at least $\k t$, showing
that for every $t\in (0,r]$, $\exp_{\c(t)}sX$ is well-defined for
$|s|\le {3\k t}$. Iterating the same argument shows that
the geodesic $\exp_{\c(t)}sX$ is actually defined for every $t\in (0,r]$ and for
every $s\in\R$.\vs

\begin{center}\input{f4.pstex_t}\vs
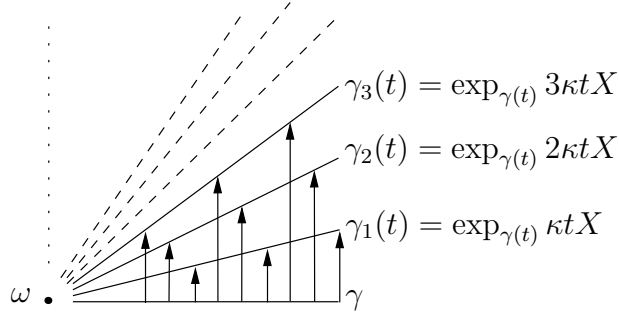
{\sc Figure 3.} Idea of the proof of
  Proposition \ref{complete}.\end{center} 

In particular, for $t=r$, $\c(t)=x$, we have proved that the geodesic through $x$
tangent to $X\in V_2$ is complete. Since $X$ was arbitrarily chosen,
the whole integral leaf of $V_2$ through $x$ is thus complete.
\r

In order to apply this result, we need to show that incomplete leaves
actually exist.

\begin{elem}\label{exist}
There exist incomplete maximal leaves $M_i$ of $V_i$ or, equivalently,
incomplete geodesics $\c_i$ tangent to $V_i$ for $i=1$ and $i=2$.
\end{elem}
\begin{proof}
Let $\c:(0,1]\to \mm$ be an incomplete geodesic, such that
$\lim_{t\to 0}\c(t)=\o$ in $(\mmm,d)$. We may assume that $\c$ is not
tangent to $V_1$ or $V_2$: If for instance $\c$ were tangent to $V_1$,
we replace it by $\c^X$ given by Lemma \ref{geo} (ii), which is
neither tangent to $V_1$ nor to $V_2$.

Let $X_1$ and $X_2$ denote the projections of $\dot\c$ on $V_1$ and
$V_2$ respectively, which are clearly parallel along $\c$. We denote
by $r_i:=|X_i|\ne 0$ the norms of $X_i$ and by $r:=\sqrt{r_1^2+r_2^2}$
the norm of $\dot\c$. Define the {\em slope} of $\c$ to be the
quotient $q(\c):=r_1/r_2$.

We claim that $\c_1(t):=\exp_{\c(t)}(-tX_1)$ is defined for every
$t\in(0,1]$, and is an incomplete geodesic tangent to $V_2$, such that
$\lim_{t\to 0}\c_1(t)=\o$ in $(\mmm,d)$. The argument is similar to
that used in the proof of Proposition \ref{complete}: The exponential,
denoted $\c^s(t)$ of $-tsX_1$ at $\c(t)$ is well-defined by
Proposition \ref{tam} for $|s|\le \frac{\k r}{r_1}$. For every fixed
$s$ in this interval, $\c^s(t)$ is an incomplete geodesic and its
slope is (see Lemma \ref{geo} (ii)) $q(\c^s)=(1-s)q(\c)$. If $\frac{\k
r}{r_1}\ge 1$, which is equivalent to $q(\c)\le (\k^{-2}-1)^{-1/2}$,
the incomplete geodesic $\c^s$ has zero slope for $s=1$, {\em i.e.} it is
tangent to $V_2$. Otherwise, we replace $\c$ by $\c^s$ with
$s=\frac{\k r}{r_1}$ and repeat this procedure. The slope of the new
geodesic is
$$q(\c^s)=\bigg(1-\frac{\k r}{r_1}\bigg)q(\c)=q(\c)-\frac{\k
r}{r_2}\le q(\c)-{\k},$$ showing that the procedure stops after a
finite number of iterations.  Since $V_1$ and $V_2$ play symmetric
r\^oles, this finishes the proof.
\r

\begin{ecor}\label{po}
If there exists an incomplete geodesic $\c$ passing through a point
$x\in\mm$ such that $\dot\c$ is neither tangent to $V_1$ nor to $V_2$,
then $x$ belongs to a complete leaf of $V_1$ which intersects an
incomplete maximal leaf of $V_2$, and to a complete leaf of $V_2$
which intersects an incomplete maximal leaf of $V_1$.
\end{ecor}

\begin{proof}The result follows directly from the proof of Lemma \ref{exist}
together with Proposition \ref{complete}.\r

\begin{elem}\label{iso}
If $M_1$ is a maximal leaf of $V_1$ which is incomplete, then the universal coverings of all
maximal leaves of $V_2$ which intersect $M_1$ are isometric.
\end{elem}
\begin{proof}
Let $M_2(x)$ denote the maximal leaf of $V_2$ through $x$. Since every
two points of $M_1$ can be joined by a broken geodesic, it is enough
to show that $M_2(\c(0))$ and $M_2(\c(1))$ are locally isometric for every
geodesic $\c:[0,1]\to M_1$. Let $y$ be any point in $M_2(\c(0))$ and $U_2$ a simply connected neighborhood of $y$ in $M_2(\c(0))$.
Consider the normal vector field $X$ on $U_2$ obtained by parallel transport of $\dot\c(0)$. Since
$M_2(\c(0))$ is complete, $y$ can be
expressed as $y=\exp_{\c(0)}(Y)$ for some $Y\in
T_{\c(0)}M_2(\c(0))$. We extend $Y$ along $\c$ by parallel
transport. By Proposition \ref{complete}, the leaves $M_2(\c(t))$ are
complete, hence $\exp_{\c(t)}(sY)$ is well-defined for every $s\in \R$
and $t\in[0,1]$.  By Lemma \ref{geo} (i) we get
$\exp_{\c(t)}(Y)=\exp_y(tX)$. The exponential of $tX$ is thus
defined for all points of $U_2$ and $t\in[0,1]$ so Lemma
\ref{exp-loc} shows that $U_2$ is isometric with its image by $\exp(X)$ in $M_2(\c(1))$. 

The very same argument actually shows that the globally defined vector field obtained by parallel transport on the universal covering of $M_2(\c(0))$ defines a map $\widetilde{M_2(\c(0))}\to M_2(\c(1))$ which is a local isometry. This map lifts to a local isometry $\widetilde{M_2(\c(0))}\to \widetilde{M_2(\c(1))}$ which is a global isometry as these two manifolds are complete and simply connected.
\r

Consider now a metric $g$ on $\mm$ obtained as the pull-back of 
a metric in the conformal class $c$ on
$M$. Let $\f$ be the conformal factor relating $g$ to $g_0$ by
$g_0=\f^2 g$ and let $\t d$ the geodesic distance induced on $\mm$ by
$g$. Denote by $B_x(r)$ and $\t B_x(r)$ the set of points at distance
less than $r$ from $x$ with respect to $d$ and $\t d$
respectively. Recall that by Lemma \ref{3.5} $\f$ is quasi-linear, so
there exist positive constants $k_1,\ k_2$ such that
\beq\label{ql}k_1\dt(x)\le \f(x)\le k_2\dt(x),\qquad \forall\
x\in\mm.\eeq

\begin{elem}\label{ball} For every $x\in\mm$ and positive real number $r$, 
the open ball $B_{x}(r)$ contains the open ball $\t B_{x}(\t r)$,
where $\t r=\frac{r}{k_2(\dt(x)+r)}$.
\end{elem}
\begin{proof}
For every $y\in \overline {B_{x}(r)}$ we have $\dt(y)\le \dt(x)+r$, so
by \eqref{ql} $\f(y)\le k_2(\dt(x)+r)$. Consequently, the $g^0$-length
$l^0(c)$ and $g$-length $l(c)$ of every path contained in $\overline
{B_{x}(r)}$ are related by \beq\label{path} l^0(c)\le
k_2(\dt(x)+r)l(c).\eeq Assume there exists $z\in \t B_{x}(\t
r)\sm {B_{x}(r)}$ and let $c:[0,1]\to \mm$ be any path joining
$x$ and $z$. Define $s_0=\inf\{s\ |\ c(s)\notin B_{x}(r)\}$ and
consider the path $c'=c|_{[0,s_0]}$ which is clearly contained in
$\overline {B_{x}(r)}$.  Then $l^0(c')\ge r$, so by \eqref{path},
$l(c)\ge l(c')\ge \t r$. Since this holds for every path $c$, we must
have $\t d(x,z)\ge \t r$, contradicting the fact that $z\in \t
B_{x}(\t r)$.
\r

\begin{elem}\label{disj}
Let $\c:(0,a]\to\mm$ be an incomplete $g_0$-geodesic
parametrized by arc-length,
such that $\lim_{t\to 0}\dt(\c(t))=0$. There exist positive real
numbers $\rho,\ q\in(0,1)$ such that the open balls
$B_n:=B_{\c(q^n)}(\rho q^n)$, $n\in\N$, are all pairwise disjoint.
\end{elem}
\begin{proof}
Recall that by Proposition \ref{tam} we have control on the distance
from $\c(t)$ to the singularity $\omega$, {\em i.e.} there exists a constant $\k\in(0,1)$,
independent of $\c$, such that:
\beq\label{expl}{\k}{t}\le\dt(\c(t))\le t,\qquad \forall\ t\in(0,a].\eeq We start with arbitrary
$\rho$ and $q$ in $(0,1)$.  For every $y\in B_n$, Equation
\eqref{expl} yields
$$ ({\k}-\rho)q^n\le\dt(\c(q^n))-\rho q^n \le \dt(y)\le
\dt(\c(q^n))+\rho q^n\le (\rho+1)q^n,$$ so by \eqref{ql} we get
$$k_1({\k}-\rho)q^n\le \f(y)\le k_2(\rho+1)q^n.$$ It is thus enough to
choose $\rho$ and $q$ such that $k_2(\rho+1)q^{n+1}<k_1({\k}-\rho)q^n$
for every $n$, which is equivalent to $\rho<{\k}$ and
$q<\frac{k_1}{k_2(\rho+1)}({\k}-\rho).$
\r

\begin{ecor}\label{inter} Consider the open subset
$B:=\cup_{n\ge 1}B_n$ in $\mm$. There exists $f\in \G$ different from
the identity such that $f(B)\cap B\ne\emptyset$.
\end{ecor}
\begin{proof}
Lemma \ref{ball} applied to $x=\c(q^n)$ and $r=\rho q^n$ shows that
$B_n$ contains the open ball $\t B_x(\t r)$ where
$$\t r=\frac{r}{k_2(\dt(x)+r)}=\frac{\rho q^n}{k_2(\dt(x)+\rho
q^n)}\ge \frac{\rho q^n}{k_2(q^n+\rho
q^n)}=\frac{\rho}{k_2(1+\rho)}.$$ Recall that $\G$ acts by isometries
on $(\mm,g)$ and that $(\mm,g)/\G=(M,g)$.  If $f(B_n)\cap
B_m=\emptyset$ for every $f\in\G$ and $m\ne n$, the projections
$\pi(B_n)$ of $B_n$ onto $M$ would be pairwise disjoint sets, each of
them containing a ball of $g$-radius $\frac{\rho}{k_2(1+\rho)}$ in
$M$. This is impossible since $M$ is compact, thus proving our
assertion.
\r

The last step in the proof of Theorem \ref{main} is the following:

\begin{elem}\label{flat}
Let $M_1$ be a maximal leaf of $V_1$ which is incomplete. Then for
every $x\in M_1$, the maximal leaf $M_2(x)$ of $V_2$ through $x$ is
flat.
\end{elem}
\begin{proof}
Let $\c:(0,1]\to M_1$ be an incomplete geodesic with respect to $g_0$
parametrized by arc-length, such that $\c(1)=x$
and $\c(t)$ converges to $\o$ (with respect to $d$) as $t$ tends to
$0$. By Lemma \ref{disj} one can find $\rho,\ q\in(0,1)$ such that the
open balls $B_n:=B_{\c(q^n)}(\rho q^n)$ are pairwise disjoint.
Moreover, one can choose $\rho$ such that each maximal leaf of $V_2$ through a
point of $B=\cup_{n\ge 1}B_n$ intersects $M_1$. Indeed, this follows
from Lemma \ref{polidisc} provided that $\rho q^n$ is smaller than
$\sigma(\c(q^n))$ for every $n$. Since $\sigma$ is quasi-linear, there
exists some $\sigma_0$ such that $\sigma(x)\ge\sigma_0 \dt(x)$, so
from \eqref{expl} it suffices to take $\rho<{\k}{\sigma_0}$.

Corollary \ref{inter} now shows that there exists $f\in \G$ different
from the identity and $y,z\in B$ such that $y=f(z)$. Then $f$ maps the
integral leaf $M_2(z)$ of $V_2$ through $z$ to the integral leaf  $M_2(y)$ of
$V_2$ through $y$. Since both leaves intersect $M_1$, Lemma \ref{iso}
shows that there exists a global isometry between their
universal coverings  $\widetilde{M_2(y)}$ and $\widetilde{M_2(z)}$.
Moreover, $f$ lifts to a strict homothety $\tilde f:\widetilde{M_2(z)}\to \widetilde{M_2(y)}$.
Composing $\tilde f$ with the isometry above, we
obtain a strict homothety of $\widetilde{M_2(z)}$. Since $\widetilde{M_2(z)}$ is complete,
Lemma 2, page 242 in \cite{kn1} shows that it must be flat. By Lemma
\ref{iso} again, all the other leaves tangent to $V_2$ must be flat as
well.
\r

We are now in position to complete the proof of Theorem \ref{main}.
From Corollary \ref{po}, and Lemma \ref{flat}, the sectional curvature
of $g_0$ vanishes at each point $x$ which belongs to an incomplete
geodesic which is neither tangent to $V_1$ nor to $V_2$. The proof of
Lemma \ref{geo} (ii) shows that the set of such points is dense in
$\mm$. Thus $(\mm,g_0)$ is a flat Riemannian manifold, so the holonomy
group of $D=\n^0$ is discrete.
\r

\obs The only place where the compactness assumption on $M$ is needed
in Theorem \ref{main}, is to ensure, by Theorem \ref{t1}, that the
minimal Riemannian cover of $(M,c,D)$ has exactly one
singularity. Theorem \ref{main} thus holds in a slightly more general
setting, and applies in particular to all Riemannian cones over complete
Riemannian manifolds. 
\eobs

\vs

\section{Stable Weyl connections}

In this section we introduce the notion of {\em stability} for Weyl
connections and show that a stable Weyl connection is necessarily
tame. As stability is an open condition in the $C^1$-topology, this
shows, in particular, that the class of tame closed Weyl connections
is significantly large.

\begin{defi}\label{at}
A Weyl connection $D$ on a conformal manifold $(M,c)$ is called {\em
stable} if there exists a complete Riemannian metric $g\in
c$ and a positive real number $\e>0$ such that
\beq\label{tame}|\theta|^2g(X,X)+(\n_X\theta)(X)\ge 2\e g(X,X),\qquad
\forall\ X\in TM,  \eeq
where $\n$ denotes the Levi-Civita
covariant derivative of $g$ and $\theta$ denotes the Lee form of $D$
with respect to $g$.
\end{defi}

Recall \cite{g} that the Lee form $\theta$ of a Weyl connection $D$ with respect
to a metric $g\in c$ measures the difference between $D$ and the
Levi-Civita connection $\n=\n^g$ of $g$: \beq\label{lee}
D_XY-\n_XY=\tilde\theta_X(Y):=\theta(X)Y+\theta(Y)X+\theta^\sharp
g(X,Y),\ \forall\ X,Y\in TM, \eeq where $\theta=g(\theta^\sharp,\cdot)$.

\begin{ere}\label{et}
With the notations from Example \ref{leit}, it is easy to see that the
standard Weyl connection $D_0$ on (a compact quotient of) a metric
cone is stable. Indeed, the Lee form of $D_0$ with respect to the
complete metric $g$  is $\theta_0:=ds$ and
it is parallel for $\n=\n^g$. Therefore
$|\theta_0|^2g+\n\theta=g$ thus \eqref{tame} is even an equality for $\e=1/2$. 
\end{ere}

\begin{ere}
An exact Weyl connection on a compact conformal manifold $M$ can not be
stable. Indeed, its Lee form $\theta$ with respect to any
metric is exact, $\theta=d\varphi$ so \eqref{tame} cannot hold at
points where $\varphi$ reaches its maximum on $M$.
\end{ere}
\smallskip

The remaining part of this section is devoted to the:

\begin{proof}[Proof of Theorem \ref{geod}] Let $\theta$ be the Lee form of $D$ with respect to
  $g$. If $\n$ denotes the Levi-Civita covariant derivative of $g$,
  then we get from (\ref{lee}), see also \cite{g}: 
\beq\label{we}D_X-\n_X=(\theta\wedge
  X)_*-p\theta(X)\id\eeq on $T^*M^{\otimes p}$, where $(\theta\wedge
  X)_*$ is the usual extension of the endomorphism $\theta\wedge X$ as
  a derivation (in fact, the right hand side of \eqref{we} is just
  $\tilde\theta_X$, acting as a derivation on $T^*M^{\otimes p}$). It
  is easy to check that $(\theta\wedge X)_*g=0$, so
  \eqref{we} yields \beq\label{we1}D_Xg=-2\theta(X)g.\eeq On the other
  hand, applying \eqref{we} to the Lee form $\theta$ itself yields
$$D_X\theta=\n_X\theta+ g(\theta,\theta) g(X,.)-2\theta(X)\theta,$$
thus showing that \eqref{tame} is equivalent to
\beq\label{tame2}(D_X\theta)(X)\ge 2\e g(X,X)-2 \theta(X)^2,\qquad
\forall\ X\in TM.  \eeq Let $\c(t)$ be a geodesic with respect to $D$
on $M$ and let $I=(a,b)$ denote its maximal domain of definition, with
$a,b\in\overline\R=\R\cup\{\pm\infty\}$. We introduce the ``speed'' and ``slope'' functions
$F(t):={g(\dot\c(t),\dot\c(t))}^{-\frac12}$ and $H(t):=\theta(\dot
\c(t))$, defined on $I$. Using \eqref{we1} we get
\beq\label{s1}F'(t)=\dot
\c(t).F(t)=-\frac12\big[(D_{\dot\c(t)}g)(\dot\c(t),\dot\c(t))\big]\big[
{g(\dot\c(t),\dot\c(t))}^{-\frac32}\big]=F(t)H(t),\eeq and from
\eqref{tame2}, \beq\label{s2}H'(t)=\dot \c(t).H(t)=(D_{\dot
\c(t)}\theta)(\dot \c(t))\ge2\e F(t)^{-2}-2H(t)^2.  \eeq

\begin{elem}\label{z}
If $b<\infty$ ({\em i.e.} $\c$ is incomplete toward the future) then
$\underline{\lim}_{t\to b}F(t)=0$.  Similarly, if $a>-\infty$, then
$\underline{\lim}_{t\to a}F(t)=0$.
\end{elem}
\begin{proof}
We have to show that the $g$-norm of the speed vector of an
incomplete $D$-geodesic cannot be bounded on $M$. Consider the
geodesic flow of $D$, viewed as a vector field on the tangent bundle
$TM$. Let $\c$ be the maximal half-geodesic with respect to $D$,
issued from some $X\in T_xM$. There exists $T>0$ such that the maximal
integral curve through $X$ of the geodesic flow is defined only for
$t< T$. Assume that the $g$-norm of $\dot\c$ is bounded:
$g(\dot\c(t),\dot\c(t))<k^2$ for all $t\in [0,T)$. Then the
corresponding integral curve is contained in the subset
$K(k,T)$ of $TM$ defined by
$$K(k,T):=\{Y_y\in TM\ |\ d(x,y)\le kT\ \hbox{and}\ g(Y,Y)\le k^2\},$$
where $d$ denotes the geodesic distance with respect to $g$. Since $g$ is complete,
the closed geodesic balls are compact, so $K(k,T)$ is a compact subset of $TM$.
On the other hand, it is well-known that an incomplete integral curve
of a vector field cannot be contained in any compact subset, thus
proving the lemma.
\r

In order to fix the ideas, we will assume from now on that
\beq\label{ass}0\in I\qquad \mbox{and}\qquad H(0)\le0\eeq (this can
always be achieved by making a translation in time and replacing
$\c(t)$ with $\c(-t)$ if necessary).

\begin{elem}\label{crit}
The function $F$ has at most one critical point. If this happens, then
the critical point is an absolute minimum of $F$, and $\c$ is
complete. Conversely, if $\c$ is complete, then $F$ has a critical
point.
\end{elem}

\begin{proof}
Let $t_0$ be a critical point of $F$. From \eqref{s1},
$H(t_0)=0$. Moreover, \eqref{s2} shows that $H'$ is strictly positive
at each point where $H$ vanishes, so actually $H$ cannot vanish more
than once. Thus $H(t)$ is positive for $t\ge t_0$ and negative for
$t\le t_0$, so $t_0$ is a global minimum of $F$. Lemma \ref{z} then
shows that $\c$ cannot be incomplete.

Conversely, assume that $\c$ is complete, {\em i.e.} $I=\R$. If $F$ has no
critical point, $H$ does not vanish, so by our assumption \eqref{ass},
$H<0$ on $\R$. $F$ is thus a decreasing positive function, so
$\lim_{t\to\infty}H(t)F(t)=\lim_{t\to\infty}F'(t)=0$. Dividing by
$H(t)^2$ in \eqref{s2} yields
$$\frac{H'(t)}{H(t)^2}>\frac{2\e}{F(t)^2H(t)^2}-2.$$ Since the right
hand side tends to infinity as $t\to\infty$, an integration shows that
$\lim_{t\to\infty}H(t)=0$.  Using \eqref{s2} again, we then see that
there exist some $t_0\in\R$ and $\delta>0$ such that $H'(t)>\delta$
for $t>t_0$. This of course contradicts the fact that $H$ is negative
on the whole real line $\R$, thus proving the lemma.
\r

The convention \eqref{ass} together with Lemma \ref{crit} ensures that
if $\c$ is an incomplete geodesic, $H$ is negative on $I$, so $F$ is
decreasing. By Lemma \ref{z}, $\c$ is complete toward $-\infty$,
{\em i.e.} $I=(-\infty,b)$, with $b\in \R_+$.

\begin{elem}\label{est}
If $\c:I\to\mm$ is a geodesic with respect to $D$ which is incomplete in the positive
direction, then \beq\label{eh}
H(t)\le-\frac{\sqrt\e}{F(t)},\qquad\forall\ t\in I.  \eeq
\end{elem}
\begin{proof}
If \eqref{eh} does not hold, there exists $t_0\in I$ such that
$$-\frac{\sqrt\e}{F(t_0)}<H(t_0)<0.$$ We define the open set
$$I':=\{t\in I\ |\ -\frac{\sqrt\e}{F(t)}<H(t)\},$$ and let $(a',b')$
be the connected component of $I'$ containing $t_0$. By \eqref{s2},
$H$ is strictly increasing on $I'$. In the other hand, we have seen
that $F$ is decreasing on $I$. If $b'\in I$, we would have
$$H(b')=-\frac{\sqrt\e}{F(b')}<-\frac{\sqrt\e}{F(t_0)}<H(t_0),$$ a
contradiction. The only possibility left is thus $b'=b$. But this is
impossible as well, since by \eqref{s1},
$$\lim_{t\to
b'}\log(F(t))=\log(F(t_0))+\int_{t_0}^{b'}H(t)dt>-\infty,$$
contradicting Lemma \ref{z}.
\r

Back to the proof of Theorem \ref{geod}, using \eqref{eh} and \eqref{s1} 
we get $F'(t)\le -\sqrt{\e}$, and thus
$$\sqrt{\e}b= \int_0^b \sqrt{\e}dt\le -\int_0^b F'(t)dt\le
F(0)=g(\dot\c(0),\dot\c(0))^{-\frac12}.$$ In other words, the
life-time of every geodesic $\c$, incomplete in the positive
direction, is bounded from above by $(\e
g(\dot\c(0),\dot\c(0)))^{-\frac12}.$ Let $K$ be any compact subset of
$TM\sm\{0\}$ and let $l(K)$ denote
$$l(K):=\inf_{X\in K}\{g(X,X)\}.$$ With the notations from Section 3,
for every $X\in \I ^D\cap K$ we have $\ll ^D(X)\le (\e
l(K))^{-\frac12}$, so $D$ is tame by Proposition \ref{tam}.
\r
The stability condition \eqref{tame} is clearly open in
the $C^1$ topology defined by the metric $g$ on the space of Weyl
connections. Therefore, Theorem \ref{main} applies to open subsets 
of the space of closed Weyl connections. 

\begin{ex}\label{gc}
Let $h_t$ be a $T$-periodic 1-parameter family of metrics on a compact manifold $N$ and let $g:=dt^2+h_t$ be the generalized cylinder metric 
defined on $M:=\R\times N$, with Levi-Civita connection $\n$. It is straightforward to check that 
\beq\label{gcm}\n_{\frac{\partial}{\partial t}}\tfrac{\partial}{\partial t}=0,\qquad \left[\tfrac{\partial}{\partial t},Y\right]=0,\qquad g(\n_Y\tfrac{\partial}{\partial t},Y)=g(\n_{\frac{\partial}{\partial t}}Y,Y)=\tfrac12 \dot h_t(Y,Y),\qquad \forall Y\in TN.
\eeq

By compactness, there exists some positive real number $s_0$ such that $s_0h_t+\tfrac12\dot h_t$ is positive definite for all $t$.
Then for every $s>s_0$, the Weyl connection $D^s$ whose Lee form with respect to $g$ is $\theta^s:=s\,dt$,
is stable. Indeed, for any $X\in TM$ written as $X=a\tfrac{\partial}{\partial t}+Y$ with $Y\in TN$, we can express the left hand term of Inequality \eqref{tame} using \eqref{gcm} as
$$|\theta|^2g(X,X)+(\n_X\theta)(X)=s^2(a^2+h_t(Y,Y))+\tfrac12 s\dot h_t(Y,Y),$$
therefore \eqref{tame} is satisfied for $2\e:=\mathrm{min}\{s_0^2,s^2-ss_0\}$. If $\Gamma$ denotes the group generated by the $g$-isometry $(t,x)\mapsto (t+T,x)$, then $D^s$ defines a closed, non-exact, stable Weyl connection on the compact manifold $M/\Gamma$. Moreover, for $\dot h_t$ large with respect to $h_t$, this connection is not $C^1$-close to a quotient of a cone (which corresponds to the case $\dot h_t\equiv 0$).
\end{ex}
\vs

\section{Examples and applications}

\subsection{Holonomy issues}
An exact Weyl connection on a conformal manifold is just the
Levi-Civita connection of some metric in the conformal class.  The
possible restricted holonomy groups of exact Weyl connections are thus
given by the Berger-Simons theorem (\cite{besse}, p. 300). The
analogous question for non-closed Weyl connections can be answered from
\cite{schwach} in the irreducible case and was studied in \cite{af} in
the reducible case. It thus remains to understand the case of closed,
non-exact Weyl connections. The next result -- which gives a complete list in
the compact case, under the assumption that the connection is tame -- is a direct consequence of Theorem \ref{main} and well known facts.

\begin{ath}\label{holo}
The restricted holonomy group of a closed, non-exact, tame Weyl connection
$D$ on a compact $n$-dimensional conformal manifold $(M,c)$ is one of
the following:
$$\SO(n),\ \UU(n/2),\ \SU(n/2),\ \Sp(n/4),\ \mathrm{G}_2\ \mbox{(for $n=7$)},\
\Spin(7)\ \mbox{(for $n=8$)},\ 0. $$ Conversely, each of the groups
listed above can be realized as the restricted holonomy of a closed,
non-exact Weyl connection on a compact conformal manifold.
\end{ath}

\begin{proof} Since $D$ is locally the Levi-Civita connection of
  metrics in the conformal class $c$, the Berger-Simons theorem
  applies. Assume first that $D$ is locally symmetric. The metric
  $g_0$ on the minimal Riemannian cover $\mm$ of $(M,c,D)$ is then
  locally symmetric. Every nontrivial homothety $f$ satisfies
  $f^*g_0=\rho(f)^2 g_0$ and preserves the Riemannian curvature tensor
  $R_0$. In particular $f^*(|R_0|^2)=\rho(f)^{-4} |R_0|^2$. On the
  other hand, $R_0$ being parallel with respect to the Levi-Civita 
  connection of $g_0$, $|R_0|^2$ is constant on
  $\mm$. Since $\rho(f)\ne 1$, this shows that $(\mm,g_0)$ is flat, so
  $\Hol_0(D)=0$.

Assuming from now on that $\Hol_0(D)\ne 0$, $D$ is irreducible by
Theorem \ref{main}, so $\Hol_0(D)$ is in the Berger list \cite{besse},
p. 301. It remains to show that $\Sp(k)\cdot \Sp(1)$ can not be
realized as the restricted holonomy group of a closed, non-exact Weyl
connection. The argument is similar to the one used above. If
$\Hol_0(D)=\Sp(k)\cdot \Sp(1)$ then the minimal cover $(\mm,g_0)$ is
{\em quaternion-K\"ahler}, therefore Einstein with non-zero Ricci
tensor $\Ric=\l g_0$ \cite{besse}.  Since the homotheties that act on $\mm$
preserve the Levi-Civita connection of $g_0$, they also preserve the Ricci
tensor. We infer that every homothety has to be an isometry, which
contradicts the fact that $D$ is not exact.

Conversely, every group in the above list can be
realized as the holonomy of a closed, non-exact Weyl connection on a
compact manifold $M=S^1\times N$, obtained as 
quotient of the Riemannian cone over a manifold $(N,g)$ endowed with special structure by a non-trivial homothety
(see \cite{baer93a} for details).

\r

Note that the case $\Hol_0(M,D)=\UU(m)$ is well-known in the
literature and corresponds to {\em locally conformally K\"ahler}
(l.c.K.) manifolds.  The l.c.K. structure constructed above on
$S^1\times N$ for every Sasakian manifold $N$ has the following
special property: There exists a metric $g$ in the conformal class
such that the Lee form of $D$ with respect to $g$ is
$\nabla^g$-parallel \cite{v}.  This special kind of l.c.K. metric is
called {\em Vaisman} metric and it is known that not every
l.c.K. structure contains such a metric in the conformal class, not even 
for a deformation of the l.c.K. conformal class (see
\cite{lebrun}, \cite{tric} for examples of l.c.K. manifolds which can
not be conformally Vaisman for topological reasons, having non-zero 
Euler characteristic, and also \cite{f}
for a classification of Vaisman structures on compact 4-manifolds).

For the other holonomy groups in the above list we have the following structure result 
(note that the tameness assumption is no longer required):
\begin{ath}\label{struct}
Let $(M,c,D)$ be a compact Weyl manifold of dimension $n>2$, such that $D$ is a closed non-exact Weyl
connection whose restricted holonomy is one of the following subgroups of $\SO(n)$:
$\SU(n/2),\ \Sp(n/4), \mathrm{G}_2\subset \SO(7)$, $\Spin(7)\subset \SO(8)$ or
$0\subset \SO(n)$. Then the following hold:
\bi
%\item There exists a compact Riemannian manifold $(N,g_N)$ satisfying one of the conditions (2)--(6) above, such that
%the universal cover $\tilde M$ of $M$, together with the metric $g_0$ whose Levi-Civita connection
%is the pull-back of $D$, is a Riemannian cone over $(N,g_N)$, {\em i.e.} $\tilde M=\R^*_+\times N$, and
%$g_0=dt^2+t^2g_N$.
\item The minimal Riemannian cover of $(M,c,D)$ is a Riemannian cone.
%over a finite quotient of $(N,g_N)$.
\item The manifold $M$, endowed with its Gauduchon metric, is a mapping torus.
% of an isometry of this finite quotient of $(N,g_N)$.
\ei
\end{ath}
\begin{proof}
Let $g\in c$ denote the {\em Gauduchon metric} of $D$ on $M$ (which is determined
up to a multiplicative constant by the fact that the Lee form of $D$ with
respect to $g$ is $\delta^g$--co-closed, see \cite{g1}), as well as its
pull-back to the  
universal cover $\tilde M$ of $M$. We denote by $g_0$ the metric on $\tilde M$ having $D$ as Levi-Civita covariant derivative.
In all five cases (2)--(6), the metric $g_0$ is Ricci-flat, so $D$ is an {\em
Einstein-Weyl} connection. This also holds on the compact manifold $M$, therefore
Theorem 3 in \cite{g} implies that the Lee form of $D$ with respect to $g$ is
parallel (and non-zero). The same is true on the complete, simply connected manifold $(\tilde
M,g)$, which is therefore isometric to a Riemannian product $(\R,ds^2)\times (N,g_N)$.
The Lee form of $D$ with respect to $g$ on $\tilde M$ is just $ds$, so
$g_0=e^{2s}g$, i.e. $g_0=dt^2+t^2g_N$ after a coordinate change $t:=e^s$. This
means that $(\tilde M,g_0)$ is the Riemannian cone over $(N,g_N)$. It is well-known,
see for example \cite{baer93a}, that if the holonomy of the Riemannian cone of
$(N,g_N)$ is one of the five groups above, then $(N,g_N)$ is Einstein with
positive scalar curvature. This, together with the fact that $N$ is closed in
$\tilde M$, (and thus complete), implies that $N$ has to be compact.

Let $f\in\pi_1(M)$ be any deck transformation, thus acting isometrically on
$(\tilde M,g)$. Since $f$ is affine with respect to $D$, it has to preserve the
Lee form of $D$ with respect to $g$, {\em i.e.} $f^*(ds)=ds$, and therefore it
preserves its $g$-dual $\d/\d s$. This means that $f$ commutes with the flow of
$\d/\d s$ on $\tilde M$, so it 
is induced by an isometry, also denoted by $f$, of $(N,g_N)$:
$f(s,x)=(s+\ln(\rho(f)),f(x))$ (recall that $\rho(f)$ is the homothety constant
of $f$ with respect to $g_0$: $f^*g_0=\rho(f)^2g_0$). It follows that the group
$\I\subset\pi_1(M)$ of deck transformations preserving $g_0$ induces a group of
isometries $\I_N$ acting freely on 
$(N,g_N)$, so the minimal Riemannian cover $(\mm,g_0)$ of $(M,c,D)$ is the
Riemannian cone over $(N,g_N)/\I_N$. 

Finally, the compactness of $N$ implies that the deck transformation group 
$\G=\pi_1(M)/\I$ of the covering $\mm\to M$ is discrete, hence isomorphic to
$\Z$, showing that $(M,g)$ is the mapping torus of an isometry of
$(N,g_N)/\I_N$.
\end{proof}

As a consequence, $\chi(M)=0$ and the fundamental group of $M$ is a finite
extension of $\Z$. Note that if $\dim M=2$ and $D$ is flat, 
its minimal covering may be $\C^*$ or $\C$. In both cases $\pi_1(M)$ is (a
finite extension of) $\Z^2$.

\obs The Berger-Simons theorem, along with the de Rham decomposition
theorem, completely classify the restricted holonomy groups of torsion-free
connections with {\em bounded} full holonomy group (as a subset of GL$(n,\R)\subset \R^{n^2}$). On the other hand,
a closed, non-exact Weyl connection is just a torsion-free connection
whose restricted holonomy group is {\em compact}, but its full holonomy
group is {\em not bounded}.  Theorem \ref{main} and the results in this
section can thus be interpreted as an holonomy classification for this
kind of connections (under the tameness assumption).  \eobs

\vs

\subsection{An example of cone-like manifold which is not tame}\label{ntame}

Let $\widehat{C_0}$ be the following rotation cone in $\R^3$:
$$\widehat{C_0}:=\{(x,y,z)\ |\ z=\sqrt{x^2+y^2}\}.$$ The set
$C_0:=\widehat{C_0}\sm\{0\}$ is a smooth Riemannian submanifold of
$\R^3$ and its metric completion is $\widehat{C_0}$. The homothety
$X\mapsto 2X$ in $\R^3$ defines by restriction a homothety $f$ of
$C_0$, which generates a group of homotheties $\G:=\{f^n\ | \
n\in\Z\}$ acting freely and properly discontinuously on $C_0$. The
quotient space is a topological torus $T^2$. The Riemannian metric on
$C_0$ defines by projection a conformal structure on $T^2$, and its
Levi-Civita connection projects to a closed, non-exact Weyl connection
on $T^2$.

We are going to apply some surgery and smoothening to get by similar
methods a closed, non-exact Weyl connection on a surface of genus 2.

To do that, consider the domain $B_0:=C_0\cap\{1<z<2\}$, remove the
two topological discs obtained as intersection of $B_0$ with the full
cylinder
$$Z_0:=\{(x,y,z)\in\R^3\ |\ y^2+(z-3/2)^2\le 1/16\},$$ connect the
borders of the two removed discs by the part of the boundary of $Z_0$
that lies {\em inside} the cone $C_0$, then smoothen it up to get a new
surface $B\subset \R^3$ such that: \bi
\item Only the part of $B_0$ inside the (larger) cylinder
$$Z:=\{(x,y,z)\in\R^3\ |\ y^2+(z-3/2)^2\le 1/8\}$$ has been changed
(in particular there are neighborhoods of the two boundary circles of
$B_0$ that are unchanged, so the gluing with the remaining part of
$C_0$ can be done smoothly);
\item The symmetries
$$S^x,S^y:\R^3\ra \R^3,\ S^x(x,y,z):=(-x,y,z),\ S^y(x,y,z):=(x,-y,z)$$
still act as isometries of $B$.  \ei The union
$$N_0:=\bigcup_{n\in\Z}f^n\left(\overline{B}\right)$$ is then a
non-closed (hence incomplete) smooth submanifold in $\R^3$ which can
be completed as a metric space by adding the origin to it. Let $g_0$
denote the induced Riemannian metric from $\R^3$. The group $\G$ acts
on $(N_0,g_0)$ by homotheties and the quotient space $N:=N_0/\G$ is a
genus 2 surface (obtained by gluing together the two circles that
constitute the boundary of $B$). The Riemannian metric $g_0$ and its
Levi-Civita connection define, by projection, a conformal structure
$c$, and a closed, non-exact Weyl connection $D$ on $N$.

\begin{center}\input{f5.pstex_t}\vs
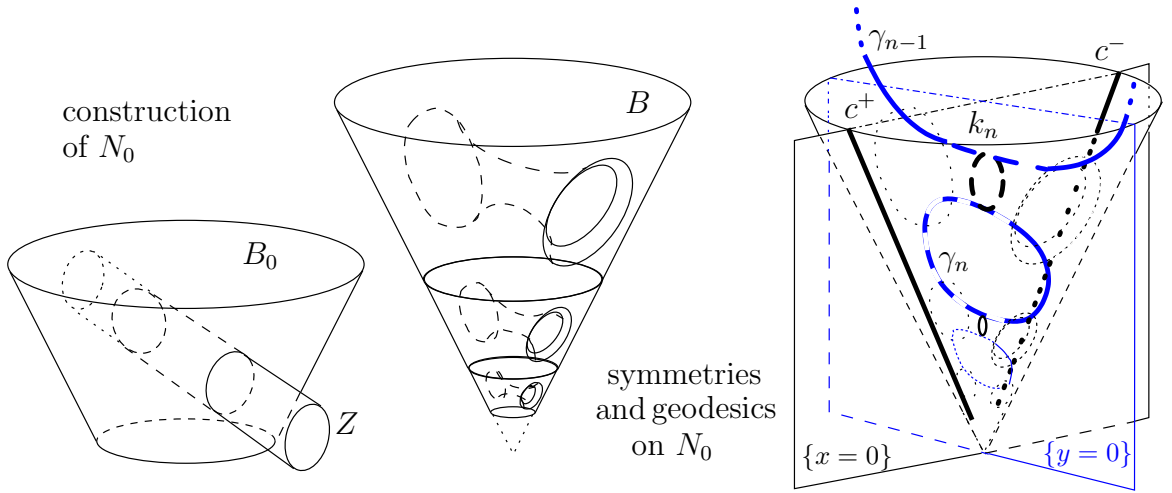
{\sc Figure 4.} Construction and 
properties of $N_0$.\end{center}

We are going to study the geodesics of $N_0$ and prove the following
\begin{prop}\label{cong}
The Weyl connection $D$ is not tame on $(N,c)$.
\end{prop}
\begin{proof} 
Consider the isometries $S^x$ and $S^y$ acting on $(N_0,g_0)$, whose
fixed point sets consist of unions of geodesics in $N_0$: \bi
\item $\mbox{Fix}(S^x)=N_0\cap\{x=0\}$, which is a union of two half
lines $c^+,c^-:(0,\infty)\ra N_0$, $c^\pm(t):=(0,\pm t,t)$ and an
infinity of circles $k_n:=\{(0,y,z)\ |\ y^2+(z-3\cdot
2^{n-1})^2=4^{n-2}\}$, $n\in\Z$;
\item $\mbox{Fix}(S^y)=N_0\cap\{y=0\}$, which is a union of closed
curves $\c_n$ connecting $f^n(B)$ with $f^{n+1}(B)$ and intersecting
$k_n$ in $P_n:=(0,0,7\cdot 2^{n-2})$, and $k_{n+1}$ in
$Q_{n+1}:=(0,0,3\cdot 2^{n-1})$.  \ei

We denote $P:=P_n$, for some positive integer $n$. The point $P$ is a
  fixed point for both isometries $S^x$ and $S^y$, and hence for their
  composition $S:=S^x\circ S^y$. The latter induces the map $X\mapsto
  -X$ on $T_PN_0$ and associates to a point $Q\in N_0$ the {\em
  geodesic reflection through $P$}, {\em i.e.} the point $\bar Q$ such that,
  for any geodesic $c^Q:(-\e,a]\ra N_0$ with $c^Q(a)=Q$ and
  $c^Q(0)=P$, $c^Q$ can be defined on a symmetric interval $[-a,a]$,
  and $\bar Q=c^Q(-a)$.

On the other hand, there exists a geodesic $\c:(-\e,T)\ra N_0$ such
that $\c(0)=P$ and $\c(t)$ tends to $\o=(0,0,0)$ as $t$ tends to
$T$. The remark above implies that the geodesic is actually defined on
$(-T,T)$ (and this is its maximal domain of definition), and
$$\lim_{t\ra T}\c(t)=\o=\lim_{t\ra -T}\c(t),$$ so both ends of the
incomplete geodesic $\c$ tend to the singularity.

For any $\e>0$, the point $\c(T-\e)$ can thus be connected by at least
two half-geodesics with $\o$, namely the two branches of $\c$, of
lengths $\e$ and $2T-\e$ respectively. As $\e$ can be chosen
arbitrarily small, we see that there is no bound for the ratios of
those lengths, therefore $N_0$ is not tame by the converse statement
in Proposition \ref{tam}.
\end{proof}

\vs

\subsection{Reducible, non-conformal, locally metric connections} We give here an
example of a non-conformal locally metric connection with reducible
holonomy which is not flat and not globally a product. 

\begin{ex} \label{exf}
 Let $(\tl M,\tl
g):=(\tl M_1,\tl g_1)\times (\tl M_2,\tl g_2)$, where
$$(\tl M_1,\tl g_1):=(\R^*_+\times S^n, dr^2+r^2 g_1)$$
is the Riemannian cone over a sphere endowed with a non-round metric
$g_1$, and $(\tl M_2,\tl g_2):=(\R,dt^2)$ is just a line.
We define now
$$\gamma_j(r,x,t):=(e^{a_j}r,x,t+ja_j),\ j=1,2,$$
where $a_1$ and $a_2$ are real numbers such that $a_1.a_2>0$ and $a_1/a_2\not\in\Q$. We see that $\gamma_j$
act by affine transformations of $(\tl M,\tl g)$. On the other hand,
the group $\Gamma$, generated by $\gamma_1$ and $\gamma_2$ is Abelian and
acts freely on $\tl M$. One can also check that this action is proper,
thus $M:=\tl M/\Gamma$ is a manifold that inherits the Levi-Civita
connection $D$ of $\tl g$ but  $D$ does not preserve any
Riemannian metric on $M$. Moreover, the $D$-stable distributions (note
that $(\tl M_1,\tl g_1)$ is an irreducible Riemannian manifold, because
$g_1$ is not the round metric) on $M$ generate transversal foliations,
but no global product structure.
\end{ex}

\subsection{Open problems}\label{55} Several natural questions emerge
from the 
considerations above. 
%We list here some of the most
%interesting ones:
 \bi
\item Theorem \ref{t1} shows that if $D$ is a closed Weyl connection on
  a compact conformal manifold $(M,c)$, then the minimal Riemannian
  cover $(\mm,g_0)$ can be metrically completed by adding exactly one
  point. The metric completion of its universal covering $(\t M,\t g)$
  is, however, not well understood: The {\em boundary} of
  $\t M$ in its metric completion may be more complicated in general,
  possibly depending on the growth of the fundamental group of $M$.
\item One can check that the stability condition forces the Lee 
form $\theta$ to be non-vanishing, therefore restricting the topology of 
$\mm$ to products $\R\times N$, in particular $\chi(M)=0$. Does the tameness 
condition also imply a topological restriction? And if this restriction 
is satisfied, is the connection automatically tame?  
\item Ultimately, does
  Theorem \ref{main} hold without the tameness assumption? 
\ei
A positive answer to this last question would be equivalent to Conjecture \ref{conj}. Note that the only crucial place where tameness is used is Proposition \ref{complete}, which ensures the existence of sufficiently many complete leaves on a reducible cone-like manifold. 
%The answer to this last question seems to be the most challenging
%problem in the holonomy theory of Weyl connections.

\end{document}

%% file: f2.pstex_t
\begin{picture}(0,0)%
\includegraphics{f2.pstex}%
\end{picture}%
\setlength{\unitlength}{3947sp}%
\begingroup\makeatletter\ifx\SetFigFont\undefined%
\gdef\SetFigFont#1#2#3#4#5{%
  \reset@font\fontsize{#1}{#2pt}%
  \fontfamily{#3}\fontseries{#4}\fontshape{#5}%
  \selectfont}%
\fi\endgroup%
\begin{picture}(4500,1374)(-74,-1573)
\put(3076,-736){\makebox(0,0)[lb]{\smash{\SetFigFont{12}{14.4}{\rmdefault}{\mddefault}{\updefault}{\color[rgb]{0,0,0}$X_{k+1}$}%
}}}
\put(3151,-1186){\makebox(0,0)[lb]{\smash{\SetFigFont{12}{14.4}{\rmdefault}{\mddefault}{\updefault}{\color[rgb]{0,0,0}$U(s_{k+2})$}%
}}}
\put(2176,-736){\makebox(0,0)[lb]{\smash{\SetFigFont{12}{14.4}{\rmdefault}{\mddefault}{\updefault}{\color[rgb]{0,0,0}$X_k$}%
}}}
\put(1276,-1336){\makebox(0,0)[lb]{\smash{\SetFigFont{12}{14.4}{\rmdefault}{\mddefault}{\updefault}{\color[rgb]{0,0,0}$U(s_k)$}%
}}}
\put(1876,-586){\makebox(0,0)[lb]{\smash{\SetFigFont{12}{14.4}{\rmdefault}{\mddefault}{\updefault}{\color[rgb]{0,0,0}$\! W_k$}%
}}}
\put(676,-361){\makebox(0,0)[lb]{\smash{\SetFigFont{12}{14.4}{\rmdefault}{\mddefault}{\updefault}{\color[rgb]{0,0,0}$U_1$}%
}}}
\put(2176,-1486){\makebox(0,0)[lb]{\smash{\SetFigFont{12}{14.4}{\rmdefault}{\mddefault}{\updefault}{\color[rgb]{0,0,0}$U(s_{k+1})$}%
}}}
\put(4051,-436){\makebox(0,0)[lb]{\smash{\SetFigFont{12}{14.4}{\rmdefault}{\mddefault}{\updefault}{\color[rgb]{0,0,0}$U'_1$}%
}}}
\put(826,-736){\makebox(0,0)[lb]{\smash{\SetFigFont{12}{14.4}{\rmdefault}{\mddefault}{\updefault}{\color[rgb]{0,0,0}$X$}%
}}}
\put(376,-736){\makebox(0,0)[lb]{\smash{\SetFigFont{12}{14.4}{\rmdefault}{\mddefault}{\updefault}{\color[rgb]{0,0,0}$V$}%
}}}
\put(-74,-1036){\makebox(0,0)[lb]{\smash{\SetFigFont{12}{14.4}{\rmdefault}{\mddefault}{\updefault}{\color[rgb]{0,0,0}$x_0=x$}%
}}}
\put(4426,-736){\makebox(0,0)[lb]{\smash{\SetFigFont{12}{14.4}{\rmdefault}{\mddefault}{\updefault}{\color[rgb]{0,0,0}$W_n$}%
}}}
\put(4426,-1036){\makebox(0,0)[lb]{\smash{\SetFigFont{12}{14.4}{\rmdefault}{\mddefault}{\updefault}{\color[rgb]{0,0,0}$x_1$}%
}}}
\put(2551,-436){\makebox(0,0)[lb]{\smash{\SetFigFont{12}{14.4}{\rmdefault}{\mddefault}{\updefault}{\color[rgb]{0,0,0}$W_{\! k+1}$}%
}}}
\end{picture}

%% file: f3.pstex_t
\begin{picture}(0,0)%
\includegraphics{f3.pstex}%
\end{picture}%
\setlength{\unitlength}{3947sp}%
\begingroup\makeatletter\ifx\SetFigFont\undefined%
\gdef\SetFigFont#1#2#3#4#5{%
  \reset@font\fontsize{#1}{#2pt}%
  \fontfamily{#3}\fontseries{#4}\fontshape{#5}%
  \selectfont}%
\fi\endgroup%
\begin{picture}(3634,1580)(7029,-5819)
\put(8326,-5161){\makebox(0,0)[lb]{\smash{\SetFigFont{12}{14.4}{\rmdefault}{\mddefault}{\updefault}{\color[rgb]{0,0,0}$c$}%
}}}
\put(8026,-5761){\makebox(0,0)[lb]{\smash{\SetFigFont{12}{14.4}{\rmdefault}{\mddefault}{\updefault}{\color[rgb]{0,0,0}$\gamma(t_0)$}%
}}}
\put(7126,-5761){\makebox(0,0)[lb]{\smash{\SetFigFont{12}{14.4}{\rmdefault}{\mddefault}{\updefault}{\color[rgb]{0,0,0}$\c$}%
}}}
\put(8026,-4936){\makebox(0,0)[lb]{\smash{\SetFigFont{12}{14.4}{\rmdefault}{\mddefault}{\updefault}{\color[rgb]{0,0,0}$x$}%
}}}
\put(7726,-4861){\makebox(0,0)[lb]{\smash{\SetFigFont{12}{14.4}{\rmdefault}{\mddefault}{\updefault}{\color[rgb]{0,0,0}$Y$}%
}}}
\put(7351,-4711){\makebox(0,0)[lb]{\smash{\SetFigFont{12}{14.4}{\rmdefault}{\mddefault}{\updefault}{\color[rgb]{0,0,0}$U$}%
}}}
\put(9376,-4561){\makebox(0,0)[lb]{\smash{\SetFigFont{12}{14.4}{\rmdefault}{\mddefault}{\updefault}{\color[rgb]{0,0,0}$\gamma^X(t)=\exp_{\c(t)}tX$}%
}}}
\put(9001,-5011){\makebox(0,0)[lb]{\smash{\SetFigFont{12}{14.4}{\rmdefault}{\mddefault}{\updefault}{\color[rgb]{0,0,0}$tX$}%
}}}
\put(9301,-5761){\makebox(0,0)[lb]{\smash{\SetFigFont{12}{14.4}{\rmdefault}{\mddefault}{\updefault}{\color[rgb]{0,0,0}$\c(t)$}%
}}}
\end{picture}

%% file: f4.pstex_t
\begin{picture}(0,0)%
\includegraphics{f4.pstex}%
\end{picture}%
\setlength{\unitlength}{3947sp}%
\begingroup\makeatletter\ifx\SetFigFont\undefined%
\gdef\SetFigFont#1#2#3#4#5{%
  \reset@font\fontsize{#1}{#2pt}%
  \fontfamily{#3}\fontseries{#4}\fontshape{#5}%
  \selectfont}%
\fi\endgroup%
\begin{picture}(3689,1950)(9666,-5971)
\put(11739,-5006){\makebox(0,0)[lb]{\smash{\SetFigFont{12}{14.4}{\rmdefault}{\mddefault}{\updefault}{\color[rgb]{0,0,0}$\c_2(t)=\exp_{\c(t)}2\k tX$}%
}}}
\put(11739,-4607){\makebox(0,0)[lb]{\smash{\SetFigFont{12}{14.4}{\rmdefault}{\mddefault}{\updefault}{\color[rgb]{0,0,0}$\c_3(t)=\exp_{\c(t)}3\k tX$}%
}}}
\put(11744,-5481){\makebox(0,0)[lb]{\smash{\SetFigFont{12}{14.4}{\rmdefault}{\mddefault}{\updefault}{\color[rgb]{0,0,0}$\c_1(t)=\exp_{\c(t)}\k tX$}%
}}}
\put(11744,-5922){\makebox(0,0)[lb]{\smash{\SetFigFont{12}{14.4}{\rmdefault}{\mddefault}{\updefault}{\color[rgb]{0,0,0}$\c$}%
}}}
\put(9666,-5906){\makebox(0,0)[lb]{\smash{\SetFigFont{12}{14.4}{\rmdefault}{\mddefault}{\updefault}$\o$}}}
\end{picture}

%% file: f5.pstex_t
\begin{picture}(0,0)%
\includegraphics{f5.pstex}%
\end{picture}%
\setlength{\unitlength}{3947sp}%
\begingroup\makeatletter\ifx\SetFigFont\undefined%
\gdef\SetFigFont#1#2#3#4#5{%
  \reset@font\fontsize{#1}{#2pt}%
  \fontfamily{#3}\fontseries{#4}\fontshape{#5}%
  \selectfont}%
\fi\endgroup%
\begin{picture}(7194,3057)(6466,-7867)
\put(8507,-7504){\makebox(0,0)[lb]{\smash{\SetFigFont{12}{14.4}{\rmdefault}{\mddefault}{\updefault}$Z$}}}
\put(7930,-6440){\makebox(0,0)[lb]{\smash{\SetFigFont{12}{14.4}{\rmdefault}{\mddefault}{\updefault}$B_0$}}}
\put(10311,-5491){\makebox(0,0)[lb]{\smash{\SetFigFont{12}{14.4}{\rmdefault}{\mddefault}{\updefault}$B$}}}
\put(12438,-5619){\makebox(0,0)[lb]{\smash{\SetFigFont{12}{14.4}{\rmdefault}{\mddefault}{\updefault}$k_n$}}}
\put(12256,-6448){\makebox(0,0)[lb]{\smash{\SetFigFont{12}{14.4}{\rmdefault}{\mddefault}{\updefault}$\c_n$}}}
\put(13226,-5160){\makebox(0,0)[lb]{\smash{\SetFigFont{12}{14.4}{\rmdefault}{\mddefault}{\updefault}$c^-$}}}
\put(11671,-5522){\makebox(0,0)[lb]{\smash{\SetFigFont{12}{14.4}{\rmdefault}{\mddefault}{\updefault}$c^+$}}}
\put(6814,-5542){\makebox(0,0)[lb]{\smash{\SetFigFont{12}{14.4}{\rmdefault}{\mddefault}{\updefault}construction}}}
\put(6814,-5767){\makebox(0,0)[lb]{\smash{\SetFigFont{12}{14.4}{\rmdefault}{\mddefault}{\updefault}of $N_0$}}}
\put(11407,-7689){\makebox(0,0)[lb]{\smash{\SetFigFont{10}{12.0}{\rmdefault}{\mddefault}{\updefault}{\color[rgb]{0,0,0}$\{x=0\}$}%
}}}
\put(11823,-5058){\makebox(0,0)[lb]{\smash{\SetFigFont{12}{14.4}{\rmdefault}{\mddefault}{\updefault}{\color[rgb]{0,0,0}$\c_{n-1}$}%
}}}
\put(12907,-7674){\makebox(0,0)[lb]{\smash{\SetFigFont{10}{12.0}{\rmdefault}{\mddefault}{\updefault}{\color[rgb]{0,0,1}$\{y=0\}$}%
}}}
\put(10476,-7403){\makebox(0,0)[lb]{\smash{\SetFigFont{12}{14.4}{\rmdefault}{\mddefault}{\updefault}{\color[rgb]{0,0,0}geodesics}%
}}}
\put(10126,-7411){\makebox(0,0)[lb]{\smash{\SetFigFont{12}{14.4}{\rmdefault}{\mddefault}{\updefault}{\color[rgb]{0,0,0}and}%
}}}
\put(10201,-7186){\makebox(0,0)[lb]{\smash{\SetFigFont{12}{14.4}{\rmdefault}{\mddefault}{\updefault}{\color[rgb]{0,0,0}symmetries}%
}}}
\put(10351,-7636){\makebox(0,0)[lb]{\smash{\SetFigFont{12}{14.4}{\rmdefault}{\mddefault}{\updefault}{\color[rgb]{0,0,0}on $N_0$}%
}}}
\end{picture}